\documentclass{article}

\usepackage[latin1]{inputenc}
\usepackage[english]{babel}
\usepackage{amsmath}
\usepackage{amsfonts}
\usepackage{amssymb}
\usepackage{bbm}
\usepackage{latexsym}
\usepackage{mathtools}
\usepackage{graphicx}
\usepackage{xcolor}
\usepackage{amsthm}
\usepackage[overload]{empheq}
\usepackage{stmaryrd}
\usepackage{caption}
\usepackage{enumerate}
\usepackage{algorithm} 
\usepackage{hyperref}
\usepackage{textcomp}
\usepackage{slashbox}
\usepackage{geometry}
\usepackage{pgfplots,pgfplotstable}

\newtheorem{theorem}{Theorem}[section]
\newtheorem{assumption}[theorem]{Assumption}
\newtheorem{proposition}[theorem]{Proposition}
\newtheorem{corollary}[theorem]{Corollary}
\newtheorem{definition}[theorem]{Definition}
\newtheorem{lemma}[theorem]{Lemma}
\newtheorem{remark}[theorem]{Remark}

\definecolor{mySkyBlue}{rgb}{0.53, 0.81, 0.92}
\definecolor{myBlue}{rgb}{0.2, 0.2, 0.6}
\definecolor{myTurquoise}{rgb}{0.19, 0.84, 0.78}
\definecolor{myOlive}{rgb}{0.42, 0.56, 0.14}
\definecolor{myForestGreen}{rgb}{0.13, 0.55, 0.13}
\definecolor{myBurntOrange}{rgb}{0.8, 0.33, 0.0}
\definecolor{myBrickRed}{rgb}{0.8, 0.25, 0.33}
\definecolor{myCerulean}{rgb}{0.11, 0.67, 0.84}
\definecolor{myMidnightBlue}{rgb}{0.1, 0.1, 0.44}

\newcommand{\Vb}{{\mathbf{b}}}

\newcommand{\Vv}{{\mathbf{v}}}

\newcommand{\Vz}{{\mathbf{z}}}

\providecommand{\Bd}{{\boldsymbol{d}}}

\providecommand{\Bn}{{\boldsymbol{n}}}

\providecommand{\Bx}{{\boldsymbol{x}}}
\providecommand{\By}{{\boldsymbol{y}}}

\newcommand{\VA}{{\mathbf{A}}}

\providecommand{\bbN}{\mathbb{N}}

\providecommand{\bbR}{\mathbb{R}}

\newcommand{\dd}{\,\mathrm{d}}
\newcommand{\nn}{\mathcal{N}\mathcal{N}}

\title{Deep neural network surrogates for non-smooth quantities of interest in shape uncertainty quantification}
\author{Laura Scarabosio\footnote{Radboud University, Institute for Mathematics, Astrophysics and Particle Physics. E-mail: scarabosio@science.ru.nl.}}

\begin{document}
\maketitle

\begin{abstract}
We consider the point evaluation of the solution to interface problems with geometric uncertainties, where the uncertainty in the obstacle is described by a high-dimensional parameter $\By\in[-1,1]^d$, $d\in\mathbb{N}$. We focus in particular on an elliptic interface problem and a Helmholtz transmission problem. Point values of the solution in the physical domain depend in general non-smoothly on the high-dimensional parameter, posing a challenge when one is interested in building surrogates. Indeed, high-order methods show poor convergence rates, while methods which are able to track discontinuities usually suffer from the so-called curse of dimensionality. For this reason, in this work we propose to build surrogates for point evaluation using deep neural networks. We provide a theoretical justification for why we expect neural networks to provide good surrogates. Furthermore, we present extensive numerical experiments showing their good performance in practice. We observe in particular that neural networks do not suffer from the curse of dimensionality, and we study the dependence of the error on the number of point evaluations (that is, the number of discontinuities in the parameter space), as well as on several modeling parameters, such as the contrast between the two materials and, for the Helmholtz transmission problem, the wavenumber.
\end{abstract}

\section{Introduction}

In many applications in engineering, the behavior of a physical system is described by a partial differential equation (PDE) involving parameters which are not known or which are subject to random variations. These uncertain quantities can be collected in a parameter $\By$ living in a parameter space $\mathcal{P}_d$ whose dimension $d$ is usually high. Since quantities of interest (QoI), such as the solution to the PDE and output functionals, depend on this parameter, it is often of practical interest to compute a \textsl{surrogate} for the parameter-to-QoI map, that is a surrogate for the map $\By\mapsto q(\By)$, $\By\in\mathcal{P}_d$, where $q$ is a quantity of interest taking values in a separable Banach space and depending on the solution of the PDE. Challenges for the computation of such surrogate are the high dimensionality of the parameter space and possible non-smooth dependence of the QoI on the parameter. In this work, we tackle a case where both these challenges occur, that is the computation of a surrogate for the parameter-to-QoI map when the parameter $\By\in\mathcal{P}_d$ describes shape variations of a random interface, the underlying PDE is either an elliptic interface problem or a Helmholtz transmission problem, and the QoI are point values of the solution at locations that can be crossed by the interface for some parameter realizations.

Uncertainties in the shape of the domain (or, as in our case, of an inclusion) can be treated with various approaches. If the uncertain variations from a deterministic shape are small compared to the size of the domain, then it is possible to use a perturbation approach based on shape calculus \cite{CherS,HaSS,HL} to construct a Taylor expansion for the parameter-to-QoI map. In this work, we consider large shape variations and therefore the perturbation approach is not applicable to our case. Other possibilities which can be applied also for large shape variations are level set methods \cite{NSM,NCSM} and the fictitious domain approach \cite{CK}. Both approaches embed the stochastic domain in a larger domain, converting a problem on a stochastic domain to an interface problem, and therefore introducing non-smooth dependence of values of the solution close to the fictitious interface with respect to the parameter representing the uncertainty. 

In this work we use the mapping technique \cite{XT,TX}: we map the configuration with stochastic interface to a configuration with deterministic interface using a diffeomorphism depending on $\By\in\mathcal{P}_d$. Such approach allows to avoid remeshing the domain for every parameter configuration and it is also convenient for the theoretical analysis. Smoothness of the parameter-to-solution map using the mapping approach has been analysed in \cite{CNT,HPS} for elliptic problems, in \cite{HSSS} for Helmholtz transmission problems, in \cite{CSZ} for stationary Navier-Stokes equations, in \cite{JSSZ} for Maxwell equations in frequency domain and in \cite{CC} for linear parabolic problems. It is known that, for interface problems with uncertain interface, point values of the solution on the physical domain, for points which are crossed by the interface for some parameter realizations, do not depend smoothly on the parameter $\By\in\mathcal{P}_d$ modeling the shape uncertainties. The reason is the non-smoothness of the solution with respect to the spatial coordinate across the interface, which interferes with the (smooth) parameter dependence of the solution on the reference configuration and causes the non-smooth dependence of the solution on the physical domain with respect to the parameter. It has been shown in \cite{LSMLMC} that, if one is not interested in a surrogate but only on statistics of point values on the physical domain, then multilevel Monte Carlo provides a robust and viable way for computing these statistics. In this work, we state instead that deep neural networks provide a robust and viable method for the construction of a surrogate of the parameter-to-QoI map when the QoI consists of point values of the solution in the physical domain.

In the literature there are methods able to deal with discontinuities in the parameter space. Among these are adaptive sparse grids \cite{MZ} and Voronoi tessellations \cite{Rushdi}. These methods though become unfeasible for high dimensions of the parameter space (say greater than 10), especially when the surface of discontinuity is not aligned with the coordinate axes in the parameter space. Another possibility would be to track the surface of discontinuity as described in \cite{ZhangGunz}. This method, though, poses restrictions on the geometry of the surface of discontinuity and seems to be feasible only for a limited number of discontinuities. 

Our numerical experiments show that deep neural networks are able to cope with high-dimensional parameter spaces and that, although their performance depends on the number of surfaces of discontinuity, they provide surrogates that have good approximation properties also when the number of discontinuities is large.

Deep neural networks have been used in the solution of PDEs for learning the solution to a single equation, for learning the physical model or, in parametric PDEs, for learning the parameter-to-QoI map. When learning the solution to a single PDE, most of the works use the equation to define the loss functional, either using the equivalence between solution of PDEs and energy minimization, as in \cite{Weinan18,Samaniego20}, of by minimizing the residual (the difference between the left and right hand side of the PDE), as in \cite{Raissi17,Lu19,Mishra20radiative}. We also mention the use of deep neural networks in fluid dynamics to detect regions in the domain where the solution to the PDE has poor regularity \cite{Ray18}. Examples of works which use deep neural networks to learn the underlying PDE to a physical problem are \cite{Berg19,Raissi18}. These works exploit the tunability of parameters of neural networks to use the latter as descriptors of the underlying physical dynamics. The work presented in this paper falls in the third category, the one of parametric PDEs, where neural networks are used to learn the mapping which associates to the value of a parameter the solution to a PDE or a QoI computed from it. Previous works in this category are \cite{KPSb,LMR,TB}. In \cite{TB}, the authors use an encoder and a one layer neural network to approximate the map from the parameter to the solution to an elliptic PDE, and they interpret the neural network construction as the recovery of a nonlinear active subspace. The work \cite{LMR} addresses applications in fluid dynamics, where a neural network is trained to approximate the map from the parameter to observables like the lift and drag coefficients. In the same paper, the authors show also that, when the dimension of the parameter space is not too high, using quasi-Monte Carlo points instead of random points for training and a smooth activation function can lead to a smaller generalization gap. The work \cite{KPSb} studies the performance of deep neural networks in approximating the solution to an elliptic PDE with affinely and non-affinely parametrized diffusion coefficients. The authors conclude that the approximation properties for deep neural networks are quite sensitive to the test case considered. The dependence of the approximation error on the dimension of the parameter space is also observed to be case-dependent, but in none of their test cases the curse of dimensionality (i.e. exponential increase of the error with the dimension) is noticed. In the framework of stochastic partial differential equations (SPDEs), we mention the works \cite{Han17,Weinan17}, where the authors use the Feynman-Kac formula to build efficient algorithms for the solution of parabolic SPDEs. 

On the theoretical side, most of the works on neural networks for parametric PDEs focus on the expessivity of networks. The work \cite{SZ} establishes dimension-independent convergence rates for the neural network approximation under the assumption that the solution to the PDE depends holomorphically on the parameter. For Bayesian inverse problems, \cite{HSZ} proves exponential expression rates for the data-to-QoI map. 
In \cite{KPSa}, using concepts from reduced basis approximation, the authors exploit the low-dimensionality of the solution manifold to obtain expression rates rates which are superior than those provided by classical neural network approximation theory. Finally, it has been shown, in \cite{Grohs18,Hutz20,Jentzen18} for the approximation error and in \cite{Berner20} for the generalization error, that neural networks are able to overcome the curse of dimensionality for some specific SPDEs. Rather than relying on these theoretical results on the approximation to the solution to a parametric PDE, in this paper we rely more on theoretical results on the approximation of piecewise smooth functions with neural networks, because we will see that our QoI is piecewise smooth with respect to the parameter representing the uncertainty. A review on theoretical approximation rates for non-smooth and piecewise smooth functions can be found in Subsection \ref{ssec:NNtheory} of this paper.

The main contributions of this work are: (1) we provide a theoretical explanation for why we expect (deep) neural networks to provide accurate surrogates for the parameter-to-QoI map, where in our case the parameter represents shape variations of an inclusion and the QoI are point values of the solution to a stochastic interface problem; (2) we give experimental evidence of the ability of neural network to perform such task; (3) we study experimentally the approximation properties in dependence of different modeling parameters, such as the dimension of the parameter space, the contrast between the two materials, the number of point evaluations and, for the Helmholtz transmission problem, the frequency.

The paper is organized as follows. In Section \ref{sec:stochinterface} we introduce the elliptic interface problem and the Helmholtz transmission problem with stochastic interface, which are our objects of study. There, we also provide a concrete parametrization of the stochastic interface and we use the mapping approach to formulate the model problems as variational problems on a configuration with deterministic interface and stochastic coefficients. Finally, we recall some facts on parameter dependence of point values of the solution on the physical configurations, which will be used in Section \ref{sec:NN}. In the latter, we first recall some basic definitions on neural networks. We proceed by reviewing the existent theoretical results on approximation of non-smooth and piecewise smooth functions, which are then applied to our case to state the approximation properties of our parameter-to-QoI map. The experimental results are presented in Section \ref{sec:numexp}, followed by some conclusions.

\section{Stochastic interface problems}\label{sec:stochinterface}

We first provide, in Subsection \ref{ssec:interface}, a parametrization of the stochastic interface. The model problems considered in this work are described in Subsection \ref{ssec:modelpb}. For both of them, we present, in Subsection \ref{ssec:mapping}, the formulation on a reference configuration, which is used in the numerical expreriments, and we recall, in Subsection \ref{ssec:paramdep}, the regularity of the mapping from the high-dimensional parameter describing the stochastic interface variations and point values of the solution.  

\subsection{Parametrization of the stochastic interface}\label{ssec:interface}

For ease of presentation, we consider a concrete parametrization of the stochastic interface, although the material of this section holds also for other interface representations. We assume that the domain inside the stochastic interface is star-shaped with respect to the origin, and we represent the interface by a parameter- and angle-dependent radius:
\begin{equation}\label{eq:radiusy}
r(\By;\varphi) = r_0(\varphi) + \sum_{j=1}^{d} b_j y_j\psi_j(\varphi),\quad \varphi\in[0,2\pi),
\end{equation}
where $r_0\in C^1_{per}([0,2\pi))$ is a deterministic interface which is perturbed by stochastic fluctuations. Namely, $(\psi_j)_{j\geq 1}$ are linearly independent functions in $C^1_{per}([0,2\pi))$ and $\By:=(y_j)_{j=1}^d\in [-1,1]^d$ is the image of $d$ independent, identically distributed (i.i.d.) random variables $\sim\mathcal{U}([-1,1])$. The latter are scaled by deterministic coefficients $(b_j)_{j\geq 1}$. As basis functions we take
\begin{equation*}
\psi_j(\varphi)=\begin{cases}
                   \sin(\tfrac{j+1}{2}\varphi)& \text{for }j \text{ odd},\\
                   \cos(\tfrac{j}{2}\varphi)& \text{for }j\text{ even},
                  \end{cases}
\end{equation*}
corresponding to the eigenfunctions of a rotation-invariant covariance kernel when $r_0$ is a circle. Equation \eqref{eq:radiusy} can be thought of as the truncation to the $d$-th term of a Karhunen-Lo\`eve-like expansion. We denote by $\mathcal{P}_d:=[-1,1]^d$ the high-dimensional parameter space. We assume the following on the expansion in \eqref{eq:radiusy}.
\begin{assumption}\label{ass:radius}
There exists a $d$-independent constant $C>0$ such that, for every $j\in\mathbb{N}$, $b_{2j},b_{2j-1}\leq Cj^{-p}$, with $p>1$. Moreover, $\sum_{j\geq 1}|b_j|\leq \frac{r_0^{-}}{2}$, with $r_0^{-}=\inf_{\varphi\in[0,2\pi)}r_0(\varphi)$.
\end{assumption}
The requirement $p>1$ ensures that the radius is continuous with $C^0$-norm bounded independently of $d$ \cite{HSSS}. In the numerical experiments, we also consider $p=1$, as limit case. The second part of the assumption ensures instead a uniform bound on the amplitude of the variations. 

\subsection{Model problems}\label{ssec:modelpb}

We denote by $\Gamma(\By):=\left\{r(\By;\varphi), \varphi\in [0,2\pi)\right\}$ the parameter-dependent interface, by $D_i=D_i(\By)$ the part of the domain which is enclosed in $\Gamma(\By)$ and by $D_o=D_o(\By)$ the part of the domain outside $\Gamma(\By)$, see Figure \ref{fig:domain}. 
We consider the two following model problems.

\begin{figure}[t]
\begin{center}
\fbox{\begin{tikzpicture}[scale=0.8]
\footnotesize
   \useasboundingbox (-3.2,-3) rectangle (3.2,3);
   \draw[dotted,thick] (0,0) circle[radius=1.6];
 \filldraw[fill=mySkyBlue!70!white,draw=mySkyBlue!40!myBlue,domain=0:6.28,samples=100]  plot ({1.6*cos(\x r)*(1+0.05*cos(3*\x r)+0.04*cos(8*\x r)+0.015*cos(11*\x r)}, {1.6*sin(\x r)*(1+0.05*cos(3*\x r)+0.04*cos(8*\x r)+0.015*cos(11*\x r)});
 \node at (0,-0.25) {$0$};
 \draw[very thin] (0,0)--(1.3,1);
 \node at (0.5,0.5) {$r(\By;\varphi)$};
  \node at (-0.4,-0.9) {$D_i(\By)$};
   \node at (0,2.2) {$D_o(\By)$};
\node at (-0.3,-1.8) {$\Gamma(\By)$};
 \end{tikzpicture}}
 \caption{Geometry of the model problems.}\label{fig:domain}
 \end{center}
  \end{figure}
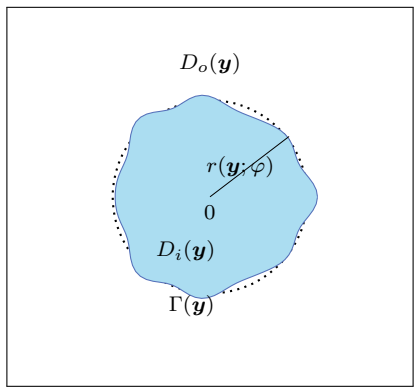
  
 \medskip 
\textbf{Elliptic interface problem}: we take $D=(-1,1)^2$, $D=D_i(\By)\cup \Gamma(\By)\cup D_o(\By)$, and we take
\begin{subequations}
    \begin{align}[left=\empheqlbrace]
   &-\nabla \cdot (\alpha(\Gamma(\By);\Bx)\nabla u_{\text{ell}}) = f,\quad\text{in } D_i(\By)\cup D_o(\By),\\
   &\llbracket u_{\text{ell}} \rrbracket_{\Gamma(\By)}=0,\quad\llbracket \alpha(\Gamma(\By);\Bx)\nabla u_{\text{ell}} \cdot \Bn \rrbracket_{\Gamma(\By)}=0,\\
   &\left.u_{\text{ell}}\right|_{\partial D}=0,\\
   &\text{for every }\By\in\mathcal{P}_d,\nonumber
    \end{align}\label{eq:ep}
    \end{subequations}
where $\partial D$ coincides with the outer boundary of $D_o$, $\Bn=\Bn(\By)$ is the normal to $\Gamma(\By)$ pointing towards $D_o$ and $\llbracket \cdot \rrbracket_{\Gamma(\By)}$ denotes the jump across $\Gamma(\By)$. We assume $f\in L^2(D)$, while $\alpha$ is a piecewise-constant coefficient which has a jump at the parameter-dependent interface:
     \begin{equation}\label{eq:alphaell}
   \alpha(\Gamma(\By);\Bx)=\begin{cases}
   1 &\text{if } \Bx\in D_o(\By),\\
   \alpha_i &\text{if } \Bx\in D_i(\By),
   \end{cases}
    \end{equation}
for some $\alpha_i>0$.

 \medskip 
\textbf{Helmholtz interface problem}: in this case the problem is posed on the whole $\mathbb{R}^2$ and $D_o(\By)=\mathbb{R}^2\setminus (\Gamma(\By)\cup D_i(\By))$ is unbounded. We study
\begin{subequations}
    \begin{align}[left=\empheqlbrace]
 &-\nabla\cdot\left(\alpha(\Gamma(\By);\Bx)\nabla u_{\text{Helm}}\right) - \kappa^2(\Gamma(\By);\Bx)u_{\text{Helm}} = 0 \quad \text{in }\mathbb{R}^2,\\
 & {\llbracket u_{\text{Helm}} \rrbracket_{\Gamma(\By)}=0,\quad\llbracket \alpha(\Gamma(\By);\Bx)\nabla u_{\text{Helm}} \cdot \Bn \rrbracket_{\Gamma(\By)}=0},\label{eq:jumps}\\
 &  \lim_{\lVert\Bx\rVert\rightarrow\infty}\left(\dfrac{\partial}{\partial \lVert\Bx\rVert}-\rm{i}\kappa\right)(u_{\text{Helm}}-u_{\text{inc}})=0,\label{eq:somm}\\
 & \text{for every }\By\in\mathcal{P}_d,\nonumber
    \end{align}\label{eq:hp}
    \end{subequations}
with piecewise-constant coefficients
     \begin{equation}\label{eq:alphahelm}
   \alpha(\Gamma(\By);\Bx)=\begin{cases}
   1 &\text{if } \Bx\in D_o(\By),\\
   \alpha_i &\text{if } \Bx\in D_i(\By),
   \end{cases}\qquad    \kappa^2(\Gamma(\By);\Bx)=\begin{cases}
   \kappa^2_o &\text{if } \Bx\in D_o(\By),\\
   \kappa^2_i &\text{if } \Bx\in D_i(\By),
   \end{cases}
    \end{equation}
with $\alpha_i, \kappa^2_i, \kappa^2_o>0$. The solution $u_{\text{Helm}}$ to \eqref{eq:hp} is the total field, while $u_{\text{inc}}=e^{\rm{i}\kappa_o \Bd\cdot\Bx}$ is an incident plane wave in the direction $\Bd$, $\lVert\Bd\rVert=1$. In \eqref{eq:jumps}, $\Bn$ denotes again the normal to $\Gamma(\By)$ pointing towards $D_o$. We note that, differently from the elliptic model problem, the solution to the Helmholtz transmission problem is complex-valued. Equation \eqref{eq:somm} is the Sommerfeld radiation condition for the scattered wave. Since for computational reasons we can only consider formulations on bounded domains, we consider a circle of fixed radius $R$ containing all realizations of the scatterer in its interior and, denoting by $D$ the domain inside this circle, we replace \eqref{eq:somm} with a condition involving the Dirichlet-to-Neumann map ($\operatorname{DtN}$) on $\partial D$ (see \cite[Sect. 6.2.3]{Ned} for its definition):
\begin{equation}
\dfrac{\partial}{\partial \Bn_R}(u_{\text{Helm}}-u_{\text{inc}}) = \operatorname{DtN}(u_{\text{Helm}}) - \operatorname{DtN}(u_{\text{inc}})\quad\text{on }\partial D,\tag{\ref{eq:hp}d}
\end{equation}
$\Bn_R$ being the outer normal to $\partial D$.

For the theoretical statements in the next section, we need a parameter-independent bound on the $H^1$-norm of the solution, that is we need to exclude the occurrence of resonances. We note that such a condition is also of practical relevance in the training of neural networks, because, when training a network with a finite number of samples, we cannot expect it to generalize well to resonant cases.

\begin{assumption}\label{ass:nontrapping}
For every $\By\in\mathcal{P}_d$ and every $d\in\mathbb{N}$, the Helmholtz transmission problem \eqref{eq:hp} is nontrapping, namely, $\lVert u(\By)\rVert_{H_{\kappa}^1(D)}^2:=\left|\nabla u\right|^2_{H^1(D)} + \left\|\kappa u\right\|^2_{L^2(D)}\leq C$ for all $\By\in\mathcal{P}_d$, with $C$ independent of $\By\in\mathcal{P}_d$.
\end{assumption}
It has been proved in \cite{MS} that the assumption above is fulfilled if the scatterer is star-shaped with respect to the origin and
\begin{equation}\label{eq:nontrapping}
\frac{\kappa_i^2}{\kappa_o^2}\leq \alpha_i,
\end{equation}
or if we assume the wavelenght to be large compared to the size of the scatterer \cite{HSSS}. In this work we assume that \eqref{eq:nontrapping} holds and we design the numerical experiments so that \eqref{eq:nontrapping} is fulfilled.

For both model problems, we consider as quantity of interest $q(\By)$ the value of the solution at some points close to the stochastic interface. Namely, for the elliptic interface problem we consider $q(\By)=\left\{u_{\text{ell}}(\By;\Bx_i)\right\}_{i=1}^{N_p}$ and for the Helmholtz transmission problem we focus on the field amplitude at some points, $q(\By)=\left\{\left|u_{\text{Helm}}(\By;\Bx_i)\right|\right\}_{i=1}^{N_p}$, $N_p\in\mathbb{N}$ denoting the number of evaluation points.

\subsection{Formulation on a deterministic geometry configuration}\label{ssec:mapping}

In order to avoid remeshing of the geometry when solving \eqref{eq:ep} and \eqref{eq:hp} numerically, for each model problem we consider a parameter-dependent diffeomorphism to a reference, deterministic geometry configuration. Mapping to a nominal geometry is also convenient to analyse the smoothness of the parameter-to-QoI map in Section \ref{ssec:paramdep}.

More precisely, we define a \textsl{nominal} interface $\hat{\Gamma}:=\left\{r_0(\varphi), \phi\in[0,2\pi)\right\}$, which is independent of $\By\in\mathcal{P}_d$, and we call \textsl{nominal configuration} the geometry configuration when then interface is $\hat{\Gamma}$. We denote by $\hat{D}_i$ and $\hat{D}_o$ the inner and outer domains in the nominal configuration, respectively (for the Helmholtz transmission problem, we use $\hat{D}_o$ to indicate the outer domain after truncation to the bounded domain $D$). Using the mapping approach introduced in \cite{TX,XT}, we take a parameter-dependent mapping $\Phi(\By): D\rightarrow D$, $\By\in\mathcal{P}_d$, from the nominal configuration to the configuration with interface $\Gamma(\By)$. We denote by $\hat{\Bx}\in D$ the coordinates in the nominal configuration and by $\Bx\in D$ the coordinates in the configuration with interface $\Gamma(\By)$. Concretely, as mapping here we consider
\begin{equation}\label{eq:phicircle}
 \Bx(\By)=\Phi(\By;\hat{\Bx})=\hat{\Bx}+\chi\left(\hat{\Bx}\right)\left(r(\By;\hat{\varphi}_{\hat{\Bx}})-r_0(\hat{\varphi}_{\hat{\Bx}})\right)\frac{\hat{\Bx}}{\lVert \hat{\Bx}\rVert},
\end{equation}
where $\hat{\varphi}_{\hat{\Bx}}:=\arg(\hat{\Bx})$ and $\chi:D\rightarrow [0,1]$ is a mollifier acting on the radial component of $\hat{\Bx}$. We assume that $\chi$ is supported on the annulus $2\pi\times\left[\frac{r_0^{-}}{4},\tilde{R}\right]$ (with $r_0^{-}:=\min_{\varphi\in[0,2\pi)}r_0(\varphi)$ and $\tilde{R}\geq \frac{3r_0^{-}}{2}$ such that the ball with radius $\tilde{R}$ is contained inside $D$ or coincides with it), is strictly increasing in $2\pi\times \left[\frac{r_0^{-}}{4}, r_0\right]$ and strictly decreasing in $2\pi\times\left[r_0,\tilde{R}\right]$. We also assume that $\chi$ has at least the same smoothness in $\overline{\hat{D}_i}$ and $\overline{\hat{D}_o}$ as the nominal radius $r_0$, and $\max\left\{\lVert\chi\rVert_{C^1(\overline{\hat{D}_i})},\lVert\chi\rVert_{C^1(\overline{\hat{D}_o})}\right\}\leq \frac{\sqrt{2}}{r_0^{-}}$. These assumptions guarantee that the mapping $\Phi$ in \eqref{eq:phicircle} is an orientation preserving diffeomorphism with the same spatial smoothness in $\overline{\hat{D}_i}$ and $\overline{\hat{D}_o}$ as the radius $r$ in \eqref{eq:radiusy}, and that the singular values of its Jacobian matrix have $\hat{\Bx}$-independent lower and upper bounds $\sigma_{\min},\sigma_{\max}$. Moreover, if $\lVert r(\By;\cdot)\rVert_{C^1_{per}([0,2\pi))}$ is bounded uniformly with respect to $\By$ and $d$, then the $C^1$-norm of $\Phi$ and the singular values of its Jacobian matrix are also bounded uniformly with respect to $\By$ and $d$. We refer to \cite[Sect. 3]{HSSS} for details.

Using the domain mapping $\Phi$, we can write the variational formulation of the model problems on the nominal configuration. We denote by $\hat{u}_{\text{ell}}:=\Phi^{\ast}u_{\text{ell}}$ and $\hat{u}_{\text{Helm}}:=\Phi^{\ast}u_{\text{Helm}}$ the solutions on the nominal configuration, $\Phi^{\ast}$ being the pullback. For the elliptic interface problem we have:
\begin{equation}
\begin{split}
 \text{Find }\hat{u}_{\text{ell}}(\By)\in H_0^1(D):&\int_{D}\hat{\alpha}(\By;\hat{\Bx})\hat{\nabla} \hat{u}_{\text{ell}}(\By)\cdot\hat{\nabla}\hat{v}\dd \hat{\Bx}= \int_{D}\hat{f}(\By;\hat{\Bx})\hat{v} \dd \hat{\Bx}, \\
 &\text{for every }\hat{v}\in H_0^1(D)\text{ and } \By\in\mathcal{P}_d,
 \end{split}\label{eq:varformep}
\end{equation}
where $\hat{\nabla}$ is the gradient with respect to $\hat{\Bx}$ and
\begin{subequations}
\begin{align}
 \hat{\alpha}(\By;\hat{\Bx}) &= D\Phi(\By)^{-1}D\Phi(\By)^{-\top}\det D\Phi(\By) \alpha(\By;\Phi^{-1}(\By;\Bx)),\label{eq:alphahat}\\
  \hat{f}(\By;\hat{\Bx}) &= f(\Phi(\By;\hat{\Bx}))\det D\Phi(\By) ,
\end{align}\label{eq:coeffshat}
\end{subequations}
$D\Phi$ being the Jacobian matrix of $\Phi$. Similarly, for the Helmholtz transmission problem we write:
\begin{equation}
\begin{split}
 \text{Find }\hat{u}_{\text{Helm}}(\By)\in H^1(D): &\int_{D}\hat{\alpha}(\By;\hat{\Bx})\hat{\nabla} \hat{u}(\By)\cdot\hat{\nabla}\hat{v}- \hat{\kappa}^2(\By;\hat{\Bx})\hat{u}_{\text{Helm}}(\By) \hat{v}\dd\hat{\Bx}-\int_{\partial D}\operatorname{DtN}(\hat{u}_{\text{Helm}}(\By))\hat{v} \dd S \hspace{-0.1cm}\\
 &= \int_{\partial D}\left(\dfrac{\partial u_i}{\partial \Bn_R}-\operatorname{DtN}(u_i)\right)\hat{v} \dd S,\\
 &\text{for every }\hat{v}\in H^1(D)\text{ and } \By\in\mathcal{P}_d,
\end{split}
 \label{eq:varformhp}
\end{equation}
with 
\begin{equation}\tag{\ref{eq:coeffshat}c}
 \hat{\kappa}^2 (\By;\hat{\Bx})= \det D\Phi(\By)\kappa^2(\By;\Phi^{-1}(\By;\Bx))
\end{equation}
and $\hat{\alpha}$ as in \eqref{eq:alphahat}.

We have thus obtained formulations where the coefficients $\hat{\alpha}$ and $\hat{\kappa}$ have a jump at the interface $\hat{\Gamma}$, for all $\By\in\mathcal{P}_d$, and we can discretize \eqref{eq:varformep} or \eqref{eq:varformhp} using finite elements on the reference configuration, without need for remeshing. 

\subsection{Parameter dependence of point evaluation}\label{ssec:paramdep}

We recall here some facts about the smoothness of the map $\By\mapsto q(\By)$, $\By\in\mathcal{P}_d$, when $q(\By)=\left\{u_{\text{ell}}(\By;\Bx_i)\right\}_{i=1}^{N_p}$ or $q(\By)=\left\{\left|u_{\text{Helm}}(\By;\Bx_i)\right|\right\}_{i=1}^{N_p}$. It can be proved that the solutions $\hat{u}_{\text{ell}}=\hat{u}_{\text{ell}}(\By)$ and $\hat{u}_{\text{Helm}}=\hat{u}_{\text{Helm}}(\By)$ on the nominal configuration depend analytically on the parameter $\By\in\mathcal{P}_d$, see \cite{CNT,HPS} for the elliptic problem and \cite{HSSS} for the Helmholtz transmission problem. However, when considering the solutions $u_{\text{ell}}=u_{\text{ell}}(\By)$ and $u_{\text{Helm}}=u_{\text{Helm}}(\By)$ on the physical domain, their regularity with respect to the parameter depends on the regularity of the map $\By\mapsto \hat{u}_{\ast}(\By)$ \textsl{and} on the regularity of $\hat{u}_{\ast}$ with respect to the spatial coordinate $\hat{\Bx}$, for $\ast=$ {\scriptsize{ell, Helm}}, respectively. Since the solutions to \eqref{eq:varformep} and \eqref{eq:varformhp} are non-smooth in space across the interface $\Gamma(\By)$, point values of $u_{\text{ell}}$ or $u_{\text{Helm}}$ have a non-smooth dependence with respect to $\By\in\mathcal{P}_d$, as we now explain.

For each point $\Bx_0$ in which we want to evaluate the solution, we denote by 
\begin{equation}\label{eq:hyperplane}
 \mathcal{P}_d^{\Gamma,\Bx_0}:=\left\{\By\in\mathcal{P}_d: \Bx_0\in \Gamma(\By)\right\}=\left\{\By\in\mathcal{P}_d: \Phi^{-1}(\By;\Bx_0)\in \hat{\Gamma}(\By)\right\}
\end{equation}
the set of realizations for which the point $\Bx_0$ lies on $\Gamma$. Crossing $\mathcal{P}_d^{\Gamma,\Bx_0}$ in $\mathcal{P}_d$ will make $\Gamma(\By)$ sweep $\Bx_0$. The set $  \mathcal{P}_d^{\Gamma,\Bx_0}$ is empty if the point is far enough from the interface so that it is always outside or always inside of it.
Because of the affine parametrization of the radius, we have the following.
\begin{lemma}\label{lem:hyperplane}
For each evaluation point $\Bx_0\in D$, $  \mathcal{P}_d^{\Gamma,\Bx_0}$ is a $(d-1)$-dimensional  hyperplane in $\mathcal{P}_d$.
\end{lemma}
The smoothness of the solutions to the elliptic and the Helmholtz interface problems with respect to the high-dimensional parameter is stated in the following two propositions, which are a refinement to Proposition 1 in \cite{LSMLMC}. There, we denote by $\mathcal{P}_{d}^{+,\Bx_0}$ and $\mathcal{P}_{d}^{-,\Bx_0}$ the two parts in which the parameter space is divided by the hyperplane $  \mathcal{P}_d^{\Gamma,\Bx_0}$.

\begin{proposition}\label{prop:contyell}
Consider $\Bx_0\in D$ such that $  \mathcal{P}_d^{\Gamma,\Bx_0}$ is not empty and let Assumption \ref{ass:radius} be fulfilled. For every $d\in\bbN$, the map $\By\mapsto u_{\text{\emph{ell}}}(\By;\Bx_0)$ from $\mathcal{P}_d$ to $\bbR$ has the following properties:
\begin{enumerate}[(i)]
\item it is H\"older continuous in $\overline{\mathcal{P}_{d}^{+,\Bx_0}}$ and $\overline{\mathcal{P}_{d}^{-,\Bx_0}}$, that is, it belongs to $C^{k,\gamma}(\overline{\mathcal{P}_{d}^{+,\Bx_0}})\cup C^{k,\gamma}(\overline{\mathcal{P}_{d}^{-,\Bx_0}})$, for every $k\geq 1$ and every $\gamma\in(0,1)$ such that $r_0\in C^{k,\gamma}_{per}([0,2\pi))$, $f\in C^{0,\gamma}(\overline{D})$ for $k=1$ and $f\in C^{k-2,\gamma}(\overline{D})$ for $k\geq 2$;
\item it is globally Lipschitz continuous, that is, it belongs to $C^{0,1}(\overline{\mathcal{P}_d})$, provided $f\in C^{0,\gamma}(\overline{D})$ for  some $\gamma\in(0,1)$.
\end{enumerate}
 If $p>2$ in Assumption \ref{ass:radius}, then $\lVert u_{\text{ell}}(\cdot;\Bx_0)\rVert_{C^{k,\gamma}(\overline{\mathcal{P}_{d}^{+,\Bx_0}})\cup C^{k,\gamma}(\overline{\mathcal{P}_{d}^{-,\Bx_0}})}$ has a $d$-independent upper bound for 
 \begin{equation}\label{eq:k}
\begin{cases}
 k=\lfloor p-1\rfloor,\text{ and } \gamma<p-1-k & \text{if } p-1 \text{ is not integer},\\
 k=p-2,\text{ and any } \gamma\in(0,1) &\text{otherwise},
\end{cases}
\end{equation}
and $ \lVert u_{\text{\emph{ell}}}(\cdot;\Bx_0)\rVert_{C^{0,1}(\overline{\mathcal{P}_{d}})}$ has a $d$-independent upper bound.
\end{proposition}
\begin{proof}
We first sketch the proof of $(i)$, referring to the proof of Proposition 1 in \cite{LSMLMC} for details. 

We write $u_{\text{ell}}(\By;\Bx_0)=\hat{u}_{\text{ell}}(\By;\Phi^{-1}(\By;\hat{x_0}))$, $\By\in \mathcal{P}_d$. It has been shown in \cite{HPS,HSSS} that $\Phi^{-1}(\By;\cdot)$ and $\hat{u}_{\text{ell}}(\By;\cdot)$ admit a holomorphic extension to the complex plane, so, for the regularity of $u_{\text{ell}}$ with respect to $\By$, it remains to check the regularity of $\hat{u}_{\text{ell}}$ with respect to the spatial variable. We also note that, if $\mathcal{P}_d^{\Gamma,\Bx_0}$ is not empty, then, due to our assumptions on the mollifier $\chi$, $\Bx_0$ and $\Phi^{-1}(\By;\Bx_0)$ are contained in a circle of radius $\tilde{R}$ inside the domain. This means that it is sufficient to prove smoothness of $\hat{u}_{\text{ell}}$ with respect to the spatial variable up to the boundary of the circle only and not in the whole domain. For every $d\in\mathbb{N}$, the sum in \eqref{eq:radiusy} is infinitely differentiable, and in particular the radius $r$ belongs to $C^{k,\gamma}_{per}([0,2\pi))$ for every $k\geq 1$ and $\gamma\in (0,1)$ such that $r_0\in C^{k,\gamma}_{per}([0,2\pi))$. It follows that the mapping $\Phi$ and its inverse belong to $C^{k,\gamma}(\overline{\hat{D}_i})\cup C^{k,\gamma}(\overline{\hat{D}_o})$, therefore $\hat{\alpha}(\By;\cdot)$ belongs to $C^{k-1,\gamma}(\overline{\hat{D}_i})\cup C^{k-1,\gamma}(\overline{\hat{D}_o})$, for every $\By\in\mathcal{P}_d$. Furthermore, given the assumptions on $f$, $\hat{f}(\By;\cdot)$ belongs to $C^{k-2,\gamma}(\overline{\hat{D}_i})\cup C^{k-2,\gamma}(\overline{\hat{D}_o})$ for $k\geq 2$ and to $C^{0,\gamma}(\overline{\hat{D}_i})\cup C^{0,\gamma}(\overline{\hat{D}_o})$ for $k=1$, for every $\By\in\mathcal{P}_d$. Proceeding as in the proof of Proposition 1 in \cite{LSMLMC}, we use Lemma 2 in \cite{Joc} and Theorem 3.1 in \cite{BGL} on $D$ to obtain a bound on $\lVert\hat{u}_{\text{ell}}(\By,\cdot)\rVert_{C^0(\overline{D})}$. We combine this with the regularity theorems, Theorem 8.33 in \cite{GT} for $k=1$ and Theorem 6.19 in \cite{GT} for $k\geq 2$, on the circular region $K_{\tilde{R}}$ of radius $\tilde{R}$, to obtain that $\hat{u}_{\text{ell}}(\By;\cdot)$ belongs to $C^{k,\gamma}(\overline{\hat{D}_i})\cup C^{k,\gamma}(\overline{\hat{D_o}\cap K_{\tilde{R}}})$, for every $\By\in\mathcal{P}_d$ (the theorems in \cite{GT} are for an elliptic problem without interface, but they can be adapted to elliptic interface problems \cite{LSMLMC}). Consequently, $\hat{u}_{\text{ell}}(\By;\Phi^{-1}(\By;\Bx_0))$ belongs to $C^{k,\gamma}(\overline{\mathcal{P}_d^{+,\Bx_0}})\cup C^{k,\gamma}(\overline{\mathcal{P}_d^{-,\Bx_0}})$. If $p>2$ in Assumption \ref{ass:radius}, then for $k$ and $\gamma$ as in \eqref{eq:k} the $C^{k,\gamma}$-norm of the radius is bounded independently of the dimension $d$ and, together with uniform coercivity, this implies a $d$-independent bound on the $C^{k,\gamma}$-norm of $\hat{u}_{\text{ell}}$ on $\hat{D}_i$ and $\hat{D}_o$ and therefore on the H\"older norm of $u_{\text{ell}}(\cdot;\Phi^{-1}(\cdot;\Bx_0))$ on $\overline{\mathcal{P}_d^{+,\Bx_0}}$ and $\overline{\mathcal{P}_d^{-,\Bx_0}}$.

For $(ii)$, we notice that, since $\mathcal{P}_d$ is a convex domain, the space $C^{0,1}(\overline{\mathcal{P}_d})$ coincides with $W^{1,\infty}(\mathcal{P}_d)$ \cite[Thm. 4.1]{H}. From part $(i)$ we have that $u_{\text{ell}}(\By;\Phi^{-1}(\By;\Bx_0))$ is $C^k$ with $k\geq 1$ on $\overline{\mathcal{P}}_d^{+,\Bx_0}$ and $\overline{\mathcal{P}}_d^{-,\Bx_0}$, so in particular it belongs to $W^{1,\infty}(\mathcal{P}_d^{+,\Bx_0})$ and $W^{1,\infty}(\mathcal{P}_d^{-,\Bx_0})$ (we remind that when we defined $r_0$ we required it to be at least $C^1$-continuous). Since $\mathcal{P}_d^{\Gamma,\Bx_0}$ is a zero measure set, we have that $\hat{u}_{\text{ell}}(\By;\Phi^{-1}(\By;\Bx_0))$ belongs to $W^{1,\infty}(\mathcal{P}_d)$, with norm bounded independently of $d$ if its H\"older norms in $\overline{\mathcal{P}^{+,\Bx_0}}$ and $\overline{\mathcal{P}_d^{-,\Bx_0}}$ are.
\end{proof}

Let $\mathcal{C}$ be the continuity constant of the bilinear form $a_p(v,w):=\langle \hat{\nabla}\hat{v},\hat{\nabla}\hat{w}\rangle - \langle \operatorname{DtN}(\hat{v}),\hat{w}\rangle_{\langle H^{-\frac{1}{2}}(\partial D),H^{\frac{1}{2}}(\partial D)\rangle}$ on $H^1(D)$, and let $\gamma_p$ be the coercivity constant of $a_p$ restricted to functions that satisfy the radiation condition. Moreover, we denote by $\sigma_{min},\sigma_{max}$ the $\By$- and $\hat{\Bx}$-independent lower and upper bounds on the singular values of $D\Phi$ (we note that \eqref{eq:radiusy} ensures a $\By$-independent bound on $\lVert r(\By)\rVert_{C^1_{per}([0,2\pi))}$ and thus on the singular values of $D\Phi$).

\begin{proposition} \label{prop:contyhelm}
Let Assumptions \ref{ass:radius} and \ref{ass:nontrapping} be fulfilled, and let us assume that we can build the mapping $\Phi$ in \eqref{eq:phicircle} such that $\frac{\sigma_{min}^4}{\sigma_{max}^4}\min\left\{\frac{1}{\alpha_i},\alpha_i\right\}\geq 1-\frac{\gamma_p}{\mathcal{C}}$.  Then, for every $d\in\bbN$, the map $\By\mapsto u_{\text{\emph{Helm}}}(\By;\Bx_0)$ from $\mathcal{P}_d$ to $\mathbb{C}$ has the following properties:
\begin{enumerate}[(i)]
\item it is H\"older continuous in $\overline{\mathcal{P}_{d}^{+,\Bx_0}}$ and $\overline{\mathcal{P}_{d}^{-,\Bx_0}}$, that is, it belongs to $C^{k,\gamma}(\overline{\mathcal{P}_{d}^{+,\Bx_0}})\cup C^{k,\gamma}(\overline{\mathcal{P}_{d}^{-,\Bx_0}})$, for every $k\geq 1$ and every $\gamma\in(0,1)$ such that $r_0\in C^{k,\gamma}_{per}([0,2\pi))$;
\item it is globally Lipschitz continuous, that is, it belongs to $C^{0,1}(\overline{\mathcal{P}_d})$;
\item if $\alpha_i=1$ in \eqref{eq:alphahat}, it belongs to $C^{1}(\overline{\mathcal{P}_d})$.
\end{enumerate}
 If $p>2$ in Assumption \ref{ass:radius}, then $\lVert u_{\text{\emph{Helm}}}(\cdot;\Bx_0)\rVert_{C^{k,\gamma}(\overline{\mathcal{P}_{d}^{+,\Bx_0}})\cup C^{k,\gamma}(\overline{\mathcal{P}_{d}^{-,\Bx_0}})}$ has a $d$-independent upper bound for $k$ and $\gamma$ as in \eqref{eq:k}; moreover, $ \lVert u_{\text{\emph{Helm}}}(\cdot;\Bx_0)\rVert_{C^{0,1}(\overline{\mathcal{P}_{d}})}$, and, for $\alpha_i=1$, $ \lVert u_{\text{\emph{Helm}}}(\cdot;\Bx_0)\rVert_{C^{1}(\overline{\mathcal{P}_{d}})}$,  have a $d$-independent upper bound.
\end{proposition}
The proof is analogous to the proof of Proposition 1 in \cite{LSMLMC} and the proof of Proposition \ref{prop:contyell} for the Lipschitz continuity. We recall from \cite{LSMLMC} that the technical assumption $\frac{\sigma_{min}^4}{\sigma_{max}^4}\min\left\{\frac{1}{\alpha_i},\alpha_i\right\}\geq 1-\frac{\gamma_p}{\mathcal{C}}$ is used to prove global continuity but not H\"older continuity in  $\mathcal{P}_{d}^{+,\Bx_0}$ and $\mathcal{P}_{d}^{-,\Bx_0}$.

Due to the non-smooth dependence of the point evaluation on the physical domains across sets of the kind \eqref{eq:hyperplane}, high-order methods for the computation of surrogates for this quantiy of interest perform very poorly, especially when considering more point evaluations simultaneously, that is when we have many surfaces of non-smoothness \cite{LSMLMC}. For this reason, in this paper we suggest to use neural networks with ReLU activation function to build surrogates for the maps $\By\mapsto u_{\ast}(\By;\Bx_0)$, $\ast=$ {\scriptsize{ell, Helm}}.

\section{Deep neural network surrogates}\label{sec:NN}

We first introduce, in Subsection \ref{ssec:NNdef}, the basic definitions and the notation for neural networks. In Subsection \ref{ssec:NNtheory}, we review the main facts about neural network approximation of non-smooth functions and their implications in the approximation of the maps $\By\mapsto q(\By)$, $\By\in\mathcal{P}_d$, when $q$ is a set of point evaluations for the solution to \eqref{eq:ep} or \eqref{eq:hp}.

\subsection{Neural network preliminaries and notation}\label{ssec:NNdef}

In this paragraph we follow the presentations in \cite{KPSa,KPSb,PV}, from which most of the following definitions are taken.

\begin{definition}
For $d$ and $L\in\bbN$, a n\emph{eural network} $\nn$ with input dimension $d$ and $L$ layers is a sequence of matrix-vector tuples
\begin{equation*}
\nn = ((\VA_1,\Vb_1),\ldots,(\VA_L,\Vb_L)),
\end{equation*}
where, for every $\ell=1,\ldots,L$, $\VA_{\ell}\in \mathbb{R}^{N_{\ell}\times N_{\ell-1}}$, $\Vb_{\ell}\in\mathbb{R}^{N_{\ell}}$ and $N_{\ell}\in\mathbb{N}$, and with $N_0:=d$ the input dimension and $N_L$ the output dimension. The elements of $\VA_{\ell}$ and $\Vb_{\ell}$, for $\ell=1,\ldots,L$, are called the \emph{weights} of the neural network.
\end{definition}
\begin{definition}
Given a neural network $\nn$, $\mathcal{P}_d\subset \mathbb{R}^d$ and a nonlinear function $\rho: \mathbb{R}\rightarrow\mathbb{R}$, we define the \emph{realization of $\nn$ with activation function $\rho$ over $\mathcal{P}_d$} as the map $\mathrm{R}_{\rho}(\nn): \mathcal{P}_d\rightarrow \mathbb{R}^{N_L}$ such that $\mathrm{R}_{\rho}(\nn)(\By)=\Vz_L$, where $\Vz_L$ is obtained by the recursion
\begin{align*}
\Vz_0 &:= \By,\\
\Vz_{\ell} &:= \rho(\VA_{\ell} \Vz_{\ell-1}+\Vb_{\ell}),\quad \ell=1,\ldots,L-1,\\
\Vz_L &:= \VA_L \Vz_{L-1}+\Vb_L.
\end{align*}
Here $\rho$ acts componentwise, that is $\rho(\Vv)_i=\rho(\text{\emph{v}}_i)$ for $i=1,2,\ldots$. We refer to $(N_0,\ldots,N_L)$ as the \emph{architecture} of $\nn$.
\end{definition}

Neural networks are called deep when $L>2$. Additionally, we refer to $N(\nn)=d +\sum_{\ell=1}^L N_{\ell}$ as the number of neurons and we denote by $M(\nn)=\sum_{\ell=1}^L \lVert\VA_{\ell}\rVert_0 + \lVert\Vb_{\ell}\rVert_0$ the number of non-zero weights. 

In this work we focus on the following activation functions.
\begin{definition}
For $\beta\in[0,1)$, $\rho_{\beta}(x):=\max\left\{\beta x,x\right\}$, for $x\in\mathbb{R}$, is the \emph{$\beta$-leaky rectified linear unit}, denoted $\beta$-LReLU. The function $\rho_0:=\max\left\{0,x\right\}$, for $x\in\mathbb{R}$, is called the \emph{rectified linear unit}, denoted ReLU.
\end{definition}

In the theoretical results in the next subsection, we restrict ourselves to the ReLU activation function. However, in the numerical experiments we use a leaky-ReLU activation function because its non-vanishing derivative for $x<0$ can help avoiding the occurrence of dead neurons in the training process \cite{KPSb}. In \cite{KPSb} it has been noted that, up to a multiplicative constant, the parameter $\beta$ of the leaky-ReLU activation function does not influence the approximation properties of neural network realizations. 

\smallskip

The approximation quality of neural network surrogates depend crucially on the choice of the weights $\VA_{\ell}, \Vb_{\ell}$, $\ell=1,\ldots,L$. These are selected through an optimization procedure in the training phase. More precisely, one first computes a set of $N_t$ input/output pairs, the \emph{training data}. In the framework described in the previous sections, our training data are parameter/QoI pairs $\left\{\By_n,q(\By_n)\right\}_{n=1}^{N_t}$, where $\By_n$ denotes the $n$-th realization of the parameter $\By$, sampled according to the multivariate uniform distribution, and $q(\By_n)=\left\{u_{\text{ell}}(\By_n;\Bx_i)\right\}_{i=1}^{N_p}$ or $q(\By_n)=\left\{\left|u_{\text{Helm}}(\By_n;\Bx_i)\right|\right\}_{i=1}^{N_p}$. Given the training set, we define the \emph{loss function} 
\begin{equation}\label{eq:loss}
\mathcal{L}_p(\VA_1,\Vb_1,\ldots,\VA_L,\Vb_L) := \frac{1}{N_t}\sum_{n=1}^{N_t} \left| q(\By_n)-R_{\rho}(\nn)(\By_n)\right|_X^p,
\end{equation}
for $1\leq p<\infty$ (for $p=\infty$, the sum over the samples is replaced by the maximum over them). In the next section and in the numerical experiments, we consider $p=2$. The error measure $\left|\cdot\right|_X$ depends on the space $X$ on which the quantity of interest takes values. In our simulations, where $X=\bbR^{N_p}$, we choose it to be the relative error in the Euclidean norm, that is, 
\begin{equation}\label{eq:l2relative}
\left| q(\By_n)-R_{\rho}(\nn)(\By_n)\right|_X^2=\frac{\lVert q(\By_n)-R_{\rho}(\nn)(\By_n)\rVert_{\mathbb{R}^{N_p}}^2}{\lVert q(\By_n)\rVert_{\mathbb{R}^{N_p}}^2}
\end{equation}
 ($\lVert\cdot\rVert_{\mathbb{R}^{N_p}}$ denoting the Euclidean norm in $\mathbb{R}^{N_p}$). The task of the training process is to optimize the neural network weights so as to minimize the loss function. This is a non-convex minimization problem over a high-dimensional set. Therefore, it is not uncommon in machine learning to add a regularization to the loss function \cite[Ch. 7]{Goodfellow}. Following \cite{KPSb}, in our numerical experiments we do not use a regularization, but we use multiple starting values for each optimization task. We note that different test cases might require different regularization parameters, and this would make comparisons between different settings more problematic. The optimization problem is solved iteratively with a gradient descent algorithm. In our simulations, we use the ADAM optimizer \cite{KL}, which is a variant of stochastic gradient descent (SGD).

\subsection{Theoretical results on approximation of non-smooth quantities of interest and application to point evaluation}\label{ssec:NNtheory}

The error between a neural network surrogate $R_{\rho}(\nn)$ and the quantity $q(\By)$ that we want to predict has three different sources: the \textsl{optimization error}, due to the iterative algorithm used to solve the non-convex optimization problem for the weights, the \textsl{approximation error}, due to the choice of the hypothesis space, that is the choice of the neural network architecture and the activation function, and the \textsl{generalization gap}, due to the fact that we have access to a finite number of realizations of $q(\By)$ to define our loss functional. In this section, we focus on the approximation error, for which specific results for the approximation of non-smooth functions are available. In particular, we review error estimates for neural network approximation of \textsl{piecewise smooth} functions defined on a \textsl{high-dimensional} domain. We then apply the estimates to our model problems.

Universal approximations theorems show that shallow neural networks are able to approximate continuous \cite{Cyb} and even only measurable \cite{Horn} functions. However, these results are not contructive and, in addition, deep neural networks are known to perform better than shallow networks with a comparable number of weights \cite{DB,Raghu,Tel}. 

More recently, Yarotsky showed in \cite{Yar18} an error estimate on the approximation of continuous functions by deep neural networks depending on the modulus of continuity of the function. More precisely, he showed that there exists a sequence of architectures with ReLU activation function and depths $O(M)$, with $M$ the number of weights, such that the error between a quantity $q$ and its best neural network approximation is bounded by 
\begin{equation*}
\lVert q - \mathcal{R}_{\rho}(\nn)\rVert_{L^{\infty}([-1,1]^d)}\leq C_1 \omega_q(C_2 M^{-\frac{2}{d}}),
\end{equation*}
for some constants $C_1$, $C_2$ possibly depending on $d$ but not on $M$ or $q$. In our case, this means that we can construct a neural network with ReLU activation function which approximates $q(\By)=u_{\text{ell}}(\By;\Bx_0)$ or $q(\By)=\left|u_{\text{Helm}}(\By;\Bx_0)\right|$ as 
\begin{equation*}
\lVert q - \mathcal{R}_{\rho}(\nn)\rVert_{L^{\infty}([-1,1]^d)}\leq C M^{-\frac{2}{d}},
\end{equation*}
thanks to the Lipschitz continuity proved in Propositions \ref{prop:contyell} and \ref{prop:contyhelm} (in the elliptic problem, provided $f$ is H\"older continuous). In the equation above, $C$ can depend on $d$. We also mention that in \cite{Yar19} the authors showed that a neural network having both ReLU and sine activation functions can approximate H\"older continuous functions to almost exponential accuracy.

However, the quantity of interest that we want approximate is not only Lipschitz continuous, but it has a precise structure, namely, it is piecewise smooth. For this reason, we rely on the results in \cite{PV} and deduce from them the consequences for our case. The theory in \cite{PV} is more general than in the context to which we apply it, in that it allows not only for jumps in the derivative but also for discontinuous dependence of the QoI with respect to the parameter. The results in \cite{PV} have been further generalized in \cite{IF}, which analyses the generalization error, allowing for noise in the realizations of the quantity of interest.

Following \cite{PV}, we first introduce a set of smooth functions which allows us to define \textsl{horizon functions}. For $B>0$ and $\beta=k+\gamma$, $k\in\mathbb{N}_0$ and $\gamma\in (0,1]$, we define 
\begin{equation}
\mathcal{F}_{\beta,d,B} = \left\{g \in C^k \left(\overline{\mathcal{P}_d}\right): \lVert g\rVert_{C^{k,\gamma}(\overline{\mathcal{P}_d})}\leq B\right\}.
\end{equation}
For $\beta=k+1$, we require $g$ and all its derivatives up to order $k$ to be Lipschitz continuous, rather than $g$ being $(k+1)$-times differentiable. In the following definition, $\mathbbm{1}_{A}$ denotes the indicator function of some set $A$.
\begin{definition}[{\cite[Def. 3.3]{PV}}]
Let $d\in\mathbb{N}$, $d\geq 2$, and let $\beta, B>0$. Furthermore, let $H:=\mathbbm{1}_{[0,\infty)\times \mathbb{R}^{d-1}}$ be the Heaviside function (acting on the first coordinate). We define
\begin{equation}\label{eq:horizonfct}
\mathcal{H}\mathcal{F}_{\beta,d,B}=\left\{g\circ T \in L^{\infty}\left([-1,1]^d\right): g(x) = H(x_1 + \tau(x_2,\ldots,x_d),x_2,\ldots,x_d), \tau\in\mathcal{F}_{\beta,d-1,B}, T\in \Pi(d,\mathbb{R})\right\},
\end{equation}
where $\Pi(d,\mathbb{R})\subset GL(d,\mathbb{R})$ is the group of permutation matrices on $\mathbb{R}^d$.
\end{definition}
We introduce then the set of domains with smooth boundaries as follows. For $s\in\mathbb{N}_0$, $d\geq 2$ and $\beta,B>0$, we define \cite{PV}
\begin{equation}
\mathcal{K}_{s,\beta,d,B}:=\left\{K\subset [-1,1]^d: \forall \By\in [-1,1]^d \;\exists g_{\By}\in\mathcal{H}\mathcal{F}_{\beta,d,B}: \mathbbm{1}_K = g_{\By} \text{ on } [-1,1]^d\cap \overline{B_{2^{-s}}}^{\lVert\cdot\rVert_{\ell^{\infty}}}\right\},
\end{equation}
where $\overline{B_{2^{-s}}}^{\lVert\cdot\rVert_{\ell^{\infty}}}$ denotes the closure of the ball of center $(0,\ldots,0)$ and radius $2^{-s}$ in the $\ell^{\infty}$-norm. Finally, we introduce the set of piecewise smooth functions as follows. For $r\in\mathbb{N}_0$, $d\geq 2$ and $\beta, B>0$, we define \cite{PV}
\begin{equation*}
\mathcal{E}_{s,\beta,d,B,P}:=\left\{g = \sum_{p=1}^P\mathbbm{1}_{K_p}\cdot g_p: g_p\in\mathcal{F}_{\beta,d,B}\text{ and } K_p\in\mathcal{K}_{s,\beta,d,B}\text{ for }p=1,\ldots,P\right\}
\end{equation*} 
(in \cite{PV}, the $g_p$, $p=1,\ldots,P$, are allowed to be slightly less smooth, but we do not need it here).
We are ready now to state the approximation result from \cite{PV} about approximation of piecewise smooth functions. In the following, by $(\sigma,\varepsilon)$-quantized weights it is meant weights which are all elements of $[-\varepsilon^{-\sigma},\varepsilon^{-\sigma}]\cap 2^{-\sigma\lceil\log_2(1/\varepsilon)\rceil}\mathbb{Z}$.
\begin{theorem}[{\cite[Cor. 3.7]{PV}}]\label{thm:pv}
Let $s\in\mathbb{N}_0, d\geq 2$ and $\beta, B>0$. There exist constants $c=c(d,\beta,s,B,P)>0$ and $\sigma=\sigma(d,\beta,s,B)\in\mathbb{N}$ such that, for all $\varepsilon\in(0,\frac{1}{2})$ and all $g\in\mathcal{E}_{s,\beta,d,B,P}$, there is a neural network $\mathcal{N}\mathcal{N}_{\varepsilon}^g$ with at most $(4+\lceil\log\beta \rceil)\cdot (12+\frac{3\beta}{d})$ layers, ReLU activation function $\rho_0$, and at most $c\varepsilon^{-2(d-1)/\beta}$ nonzero, $(\sigma,\varepsilon)$-quantized weights, such that
\begin{equation*}
\lVert R_{\rho_0}(\mathcal{N}\mathcal{N}_{\varepsilon}^g)-g \rVert_{L^2([-1,1]^d)}\leq \varepsilon \quad\text{ and }\quad \lVert R_{\rho_0}(\mathcal{N}\mathcal{N}_{\varepsilon}^g)\rVert_{L^{\infty}([-1,1]^d)}\leq \lceil B\rceil.
\end{equation*}
\end{theorem}

In \cite{PV}, this convergence rate has been proved to be optimal up to a logarithmic factor.

\smallskip

We now look at the consequences of Theorem \ref{thm:pv} when using neural networks to approximate $q(\By)=\left\{u_{\text{ell}}(\By;\Bx_i)\right\}_{i=1}^{N_p}$ or $q(\By)=\left\{\left|u_{\text{Helm}}(\By;\Bx_i)\right|\right\}_{i=1}^{N_p}$. 

We first consider one point evaluation, that is $N_p=1$. In this case, we have the following result.
\begin{lemma}
Let $q(\By)=u_{\text{ell}}(\By;\Bx_0)$ or $q(\By)=\left|u_{\text{Helm}}(\By;\Bx_0)\right|$, for a point $\Bx_0$ such that $\mathcal{P}_d^{\Gamma,\Bx_0}$ is not empty, and let the assumptions of Propositions \ref{prop:contyell} and \ref{prop:contyhelm} be fulfilled. Then $q\in \mathcal{E}_{s,\beta,d,B,2}$, for some $s=s(\Bx_0)\in\mathbb{N}_0, B>0$ and for any $\beta=k+\gamma>0$ such that $r_0\in\mathcal{C}^{k,\gamma}_{per}([0,2\pi))$ and, for the elliptic case, $f$ is as in Proposition \ref{prop:contyell} $(i)$. If $p> 2$ in Assumption \ref{ass:radius}, then $q\in \mathcal{E}_{s,\beta,d,B,2}$ for $\beta=k+\gamma$ with $k,\gamma $ as in \eqref{eq:k} and $B$  independent of the dimension $d$.
\end{lemma}
\begin{proof}
According to Lemma \ref{lem:hyperplane}, we have that $\tau$ in \eqref{eq:hyperplane} is a hyperplane, and therefore it belongs to $\mathcal{H}\mathcal{F}_{\beta,d,B}$, for every $\beta>0$. The constant $B>0$ is independent of $d$ if, in Assumption \ref{ass:radius}, $p> 2$, and by taking $\beta=k+\gamma$ with $k,\gamma$ as in \eqref{eq:k}. We have then that, for both model problems, $q(\By)= q_{+}(\By) \mathbbm{1}_{ \mathcal{P}_{d}^{+,\Bx_0}} +  q_{-}(\By) \mathbbm{1}_{ \mathcal{P}_{d}^{-,\Bx_0}}$, where $\mathcal{P}_{d}^{+,\Bx_0}$ and $\mathcal{P}_{d}^{-,\Bx_0}$ are as in Subsection \ref{ssec:paramdep}. 

Thanks to Proposition \ref{prop:contyell} $(i)$ for the elliptic problem and Proposition \ref{prop:contyhelm} $(i)$ for the Helmholtz problem, $q_{+}$ and $q_{-}$ are H\"older continuous in $\mathcal{P}_{d}^{+,\Bx_0}$ and $\mathcal{P}_{d}^{-,\Bx_0}$ respectively, for every $\beta=k+\gamma>0$ such that $r_0\in\mathcal{C}^{k,\gamma}_{per}([0,2\pi))$ and, for the elliptic case, $f$ is as in Proposition \ref{prop:contyell} $(i)$. In particular, if $p> 2$ in Assumption \ref{ass:radius}, $q_{+}$ and $q_{-}$ belong, respectively, to $C^{k,\gamma}(\overline{\mathcal{P}_d^{+,\Bx_0}})$ and $C^{k,\gamma}(\overline{\mathcal{P}_d^{-,\Bx_0}})$ for $k$ and $\gamma$ as in \eqref{eq:k}, with their H\"older norms bounded independently of the dimension $d$. This means that $q\in \mathcal{E}_{r,\beta,d,B,2}$ and that $B$ can be chosen to be independent of $d$ for $k$ and $\gamma$ as in \eqref{eq:k}, under the condition $p> 2$.
\end{proof}

By applying Theorem \ref{thm:pv} to our case, we have the following result.

\begin{corollary}\label{cor:onept}
Let $d\geq 2$ and let $q(\By)=u_{\text{\emph{ell}}}(\By;\Bx_0)$ or $q(\By)=\left|u_{\text{\emph{Helm}}}(\By;\Bx_0)\right|$, with $\Bx_0$ such that $\mathcal{P}_d^{\Gamma,\Bx_0}$ is not empty. Furthermore, let assumptions of Proposition \ref{prop:contyell} and \ref{prop:contyhelm} be fulfilled. For all $\beta=k+\gamma>0$ such that $r_0\in\mathcal{C}^{k,\gamma}_{per}([0,2\pi))$ and, for the elliptic case, $f$ is as in Propositions \ref{prop:contyell} $(i)$, there exist constants $c=c(d,\beta,\Bx_0,B)>0$ and $\sigma=\sigma(d,\beta,\Bx_0,B)\in\mathbb{N}$ such that, for all $\varepsilon\in(0,\frac{1}{2})$, there is a neural network $\mathcal{N}\mathcal{N}_{\varepsilon}^q$ with at most $(4+\lceil\log\beta \rceil)\cdot (12+\frac{3\beta}{d})$ layers, ReLU activation function $\rho_0$, and at most $c\varepsilon^{-2(d-1)/\beta}$ nonzero, $(\sigma,\varepsilon)$-quantized weights, such that
\begin{equation*}
\lVert R_{\rho}(\mathcal{N}\mathcal{N}_{\varepsilon}^q)-q \rVert_{L^2([-1,1]^d)}\leq \varepsilon \quad\text{ and }\quad \lVert R_{\rho_0}(\mathcal{N}\mathcal{N}_{\varepsilon}^q)\rVert_{L^{\infty}([-1,1]^d)}\leq \lceil B\rceil.
\end{equation*}
If $p> 2$ in Assumption \ref{ass:radius}, then the result above holds for $\beta=k+\gamma$ and $k,\gamma$ as in \eqref{eq:k}, with a bound $\lVert R_{\rho_0}(\mathcal{N}\mathcal{N}_{\varepsilon}^q)\rVert_{L^{\infty}([-1,1]^d)}\leq \lceil B\rceil$ independent of the dimension $d$.
\end{corollary}

When we consider more point evaluations, that is when $q(\By)=\left\{u_{\text{ell}}(\By;\Bx_i)\right\}_{i=1}^{N_p}$ or $q(\By)=\left\{\left|u_{\text{Helm}}(\By;\Bx_i)\right|\right\}_{i=1}^{N_p}$ with $N_p>1$, we can first consider one neural network for each evaluation point, and then we construct the neural network to approximate $q$ by putting all these neural network in parallel, see Figure \ref{fig:nnparallel}. This leads to the following.

\begin{corollary}\label{cor:morepts}
Let $d\geq 2$ and let $q(\By)=\left\{u_{\text{\emph{ell}}}(\By;\Bx_i)\right\}_{i=1}^{N_p}$ or $q(\By)=\left\{\left|u_{\text{\emph{Helm}}}(\By;\Bx_i)\right|\right\}_{i=1}^{N_p}$. Then the same claim as in Corollary \ref{cor:onept} holds for the approximation of $q$ but with a number of nonzero weights which is given by $cN_p\varepsilon^{-2(d-1)/\beta}$ and constants $c=c\left(d,\beta,\left\{\Bx_i\right\}_{i=1}^{N_p},B\right)$, $\sigma=\sigma\left(d,\beta,\left\{\Bx_i\right\}_{i=1}^{N_p},B\right)$.
\end{corollary}

\begin{remark}
We have stated the results on neural network approximation properties in the $L^2$-norm. In the numerical experiments, however, we use the relative $L^2$-error \eqref{eq:l2relative}. We note that, provided the denominator in the relative error is bounded from below and above over $\By\in\mathcal{P}_d$, as it turns out to be in our experiments, the two error measures are equivalent and therefore we can expect the estimates for the $L^2$-norm to hold also for the relative $L^2$-error.
\end{remark}

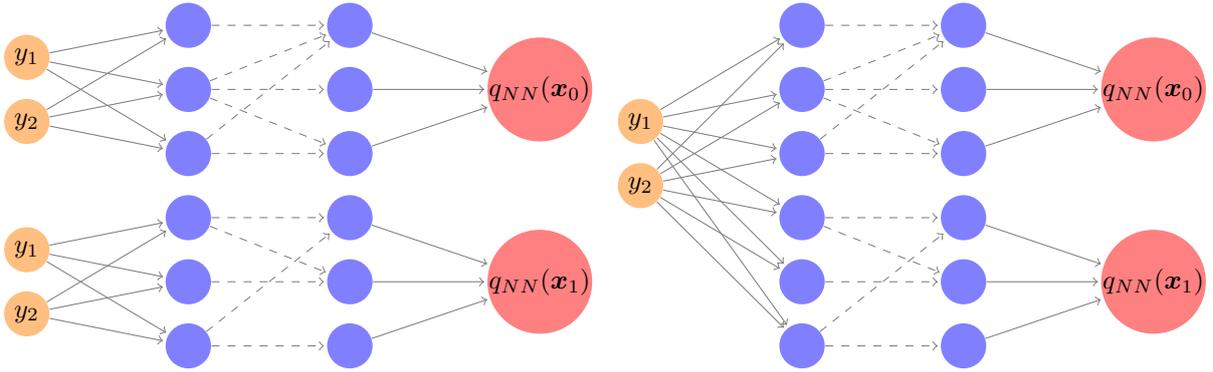
\begin{figure}[t]
\centering
\def\layersep{2.5cm}
\begin{tikzpicture}[shorten >=1pt,->,draw=black!50, node distance=\layersep,scale=0.85]
    \tikzstyle{every pin edge}=[<-,shorten <=1pt]
    \tikzstyle{neuron}=[circle,fill=black!25,minimum size=17pt,inner sep=0pt]
    \tikzstyle{input neuron}=[neuron, fill=orange!50];
    \tikzstyle{output neuron}=[neuron, fill=red!50];
    \tikzstyle{hidden neuron}=[neuron, fill=blue!50];
    \tikzstyle{annot} = [text width=4em, text centered]

    \foreach \name / \y in {1,...,2}
        \node[input neuron] (I-\name) at (0,-\y) {$y_{\y}$};

    \foreach \name / \y in {1,...,3}
        \path[yshift=0.5cm]
            node[hidden neuron] (H-\name) at (\layersep,-\y cm) {};
            
    \foreach \source in {1,...,2}
        \foreach \dest in {1,...,3}
            \path (I-\source) edge (H-\dest);
            
     \path[yshift=0.5cm]
     node[hidden neuron] (H-4) at (\layersep+\layersep,-1 cm) {};
     \path[yshift=0.5cm]
     node[hidden neuron] (H-5) at (\layersep+\layersep,-2 cm) {};
     \path[yshift=0.5cm]
     node[hidden neuron] (H-6) at (\layersep+\layersep,-3 cm) {};
            
    \path[dashed] (H-1) edge (H-4);
    \path[dashed] (H-2) edge (H-4);
    \path[dashed] (H-2) edge (H-5);
    \path[dashed] (H-2) edge (H-6);
    \path[dashed] (H-3) edge (H-4);
    \path[dashed] (H-3) edge (H-6);

    \node[output neuron, right of=H-5] (O) {$q_{NN}(\Bx_0)$};

    \foreach \source in {4,...,6}
        \path (H-\source) edge (O);

    \foreach \name / \y in {1,...,2}
        \node[input neuron] (I-\name) at (0,-\y-3) {$y_{\y}$};

    \foreach \name / \y in {1,...,3}
        \path[yshift=0.5cm]
            node[hidden neuron] (H-\name) at (\layersep,-\y-3) {};
            
    \foreach \source in {1,...,2}
        \foreach \dest in {1,...,3}
            \path (I-\source) edge (H-\dest);
            
     \path[yshift=0.5cm]
     node[hidden neuron] (H-4) at (\layersep+\layersep,-4) {};
     \path[yshift=0.5cm]
     node[hidden neuron] (H-5) at (\layersep+\layersep,-5) {};
     \path[yshift=0.5cm]
     node[hidden neuron] (H-6) at (\layersep+\layersep,-6) {};
            
    \path[dashed] (H-1) edge (H-4);
    \path[dashed] (H-1) edge (H-5);
    \path[dashed] (H-2) edge (H-5);
    \path[dashed] (H-3) edge (H-4);
    \path[dashed] (H-3) edge (H-6);

    \node[output neuron, right of=H-5] (O) {$q_{NN}(\Bx_1)$};

    \foreach \source in {4,...,6}
        \path (H-\source) edge (O);
        
    \foreach \name / \y in {1,...,2}
        \node[input neuron] (I-\name) at (9.5,-\y-1) {$y_{\y}$};

    \foreach \name / \y in {1,...,3}
        \path[yshift=0.5cm]
            node[hidden neuron] (H-\name) at (12,-\y cm) {};
            
    \foreach \source in {1,...,2}
        \foreach \dest in {1,...,3}
            \path (I-\source) edge (H-\dest);
            
     \path[yshift=0.5cm]
     node[hidden neuron] (H-4) at (14.5,-1 cm) {};
     \path[yshift=0.5cm]
     node[hidden neuron] (H-5) at (14.5,-2 cm) {};
     \path[yshift=0.5cm]
     node[hidden neuron] (H-6) at (14.5,-3 cm) {};
            
    \path[dashed] (H-1) edge (H-4);
    \path[dashed] (H-2) edge (H-4);
    \path[dashed] (H-2) edge (H-5);
    \path[dashed] (H-2) edge (H-6);
    \path[dashed] (H-3) edge (H-4);
    \path[dashed] (H-3) edge (H-6);

    \node[output neuron, right of=H-5] (O) {$q_{NN}(\Bx_0)$};

    \foreach \source in {4,...,6}
        \path (H-\source) edge (O);

    \foreach \name / \y in {1,...,3}
        \path[yshift=0.5cm]
            node[hidden neuron] (H-\name) at (12,-\y-3) {};
            
    \foreach \source in {1,...,2}
        \foreach \dest in {1,...,3}
            \path (I-\source) edge (H-\dest);
            
     \path[yshift=0.5cm]
     node[hidden neuron] (H-4) at (14.5,-4) {};
     \path[yshift=0.5cm]
     node[hidden neuron] (H-5) at (14.5,-5) {};
     \path[yshift=0.5cm]
     node[hidden neuron] (H-6) at (14.5,-6) {};
            
    \path[dashed] (H-1) edge (H-4);
    \path[dashed] (H-1) edge (H-5);
    \path[dashed] (H-2) edge (H-5);
    \path[dashed] (H-3) edge (H-4);
    \path[dashed] (H-3) edge (H-6);

    \node[output neuron, right of=H-5] (O) {$q_{NN}(\Bx_1)$};

    \foreach \source in {4,...,6}
        \path (H-\source) edge (O);
\end{tikzpicture}
\caption{Example of parallelization of two neural newtorks, for $d=2$ and $2$ point evaluations: we consider one neural network for point evaluation (left) and, for more point evaluations, we put the neural networks in parallel (right). Here $q_{NN}(\Bx_i)$, $i=0,1$, is a short notation for the neural network realization $\mathcal{R}_{\rho}\left(\nn_{\varepsilon}^{q}\right)$ when $q$ consists of two point evaluations.} 
\label{fig:nnparallel}
\end{figure}

\section{Numerical experiments}\label{sec:numexp}

In this section we examine numerically the approximation properties of some neural networks for $q(\By)=\left\{u_{\text{ell}}(\By;\Bx_i)\right\}_{i=1}^{N_p}$ or $q(\By)=\left\{\left|u_{\text{Helm}}(\By;\Bx_i)\right|\right\}_{i=1}^{N_p}$. To this aim, we fix the neural network architecture for all test cases, and we investigate the dependence of the approximation error on the following aspects.
\begin{itemize}
\item \textsl{Dimensionality}: we analyse how the error behaves in dependence of the dimension $d$ of the parameter space and on the decay $p$ of the scaling constants in the radius expansion (see Assumption \ref{ass:radius}). We will consider a circular nominal radius $r_0$ and coefficients in \eqref{eq:radiusy} given by $b_{2j-1}=b_{2j}=0.08 r_0 j^{-p}$, $j=1,\ldots,\frac{d}{2}$, and we vary $d=8,16,32,64$ and $p=1,2,3$. The maximal shape variations for the different choices of $d$ and $p$, computed as $\frac{\sqrt{2}}{r_0}\sum_{j=1}^{d/2} |b_{2j}|$, are shown in Table \ref{tab:maxshapevar}. 
It has been noted in frameworks slightly different from ours \cite{KPSb,KPSa,PV} that the approximation error of neural networks depends on an \textsl{intrinsic dimensionality} rather than from the nominal dimensionality of the problem. Motivate by this, we expect that the approximation error will possibly increase with increasing $d$, but with different rates according to the value of $p$. 
\item \textsl{Model problem} and \textsl{modeling parameters}: we expect, as also noted in other frameworks \cite{KPSb}, a different performance of the neural network depending on the model, in particular, on the following:
\begin{itemize}
\item \textsl{model problem}, whether \eqref{eq:ep} or \eqref{eq:hp}, and, for \eqref{eq:hp}, the wavenumber: the solution is oscillatory in the Helmholtz transmission problem, and the larger the wavenumber, the more oscillations it has. The more the solution oscillates in the physical space, the more oscillatory is the dependence of the point evaluation with respect to $\By\in\mathcal{P}_d$. For these reasons, for a fixed neural network architecture, we expect a larger approximation error in the Helmholtz problem than in the elliptic problem, and for the first one we expect the error to increase with increasing wavenumber;
\item \textsl{contrast}: the larger is $\alpha_i$ in \eqref{eq:alphaell} and \eqref{eq:alphahelm}, the sharper is the jump in the first derivative across $\mathcal{P}_{d}^{\Gamma,\Bx_0}$ of the point evaluation for some point $\Bx_0$; for this reason, we study whether and how the approximation properties of the neural network depend on the height of this jump;
\item \textsl{type of discontinuity}: in the Helmholtz transmission problem, if we take $\alpha_i=1$ and $\kappa_o^2\neq \kappa_i^2$, then the point evaluation is $C^1$-continuous with respect to the parameter $\By$, see Proposition \ref{prop:contyhelm} and, the larger is the ratio $\frac{\kappa_i^2}{\kappa_o^2}$, the larger is its jump in the second derivative across the surfaces of discontinuity as \eqref{eq:hyperplane}; this is why in the Helmholtz transmission problem we also study the performance of the neural network depending on the ratio $\frac{\kappa_i^2}{\kappa_o^2}$ (making sure that \eqref{eq:nontrapping} is fulfilled to have a nontrapping problem).
\end{itemize}
\item \textsl{Number and location of surfaces of non-smoothness}: the larger is the number of evaluation points $N_p$, the larger is the number of surfaces of non-smoothness \eqref{eq:hyperplane} that need to approximated; therefore, we consider $N_p=1,4,8,16,64$ points evaluations, corresponding to the evaluations points $\Bx_i=r_0(\cos(2\pi\frac{i}{N_p}),\sin(2\pi\frac{i}{N_p}))$, $i=0,\ldots,N_p-1$, and we compare the results by what predicted by Corollary \ref{cor:morepts}. We note that the points chosen are on the circle with radius the nominal radius $r_0$, and therefore they are crossed by the stochastic interface for many realizations of the parameter $\By\in\mathcal{P}_d$. Moreover, for the elliptic case and one point evaluation, we investigate whether the location of the evaluation point affects the performance of the neural network.
\end{itemize}

\begin{table}
\centering
\begin{tabular}{||c|| c |c| c| c||} 
 \hline
\backslashbox{$\boldsymbol{p}$}{$\boldsymbol{d}$}  &\textbf{8} & \textbf{16} & \textbf{32} & \textbf{64} \\ [0.5ex] 
 \hline\hline
$\mathbf{1}$ & $ 23.55 \%$ & $ 30.75 \%$ & $ 38.25 \%$ & $ 45.92 \%$ \\
 \hline
$\mathbf{2}$& $ 16.11 \%$ & $ 17.28 \%$ & $ 17.93 \%$ &  $ 18.26 \%$\\
  \hline
$\mathbf{3}$ & $13.32 \%$  & $ 13.52 \%$  & $  13.58 \%$  & $13.60 \%$\\ 
  \hline
 \end{tabular}\caption{Maximal shape variations with respect to $r_0$ for different decays $p$ of the coefficients in \eqref{eq:radiusy} and different dimensions $d$ of the parameter space. The maximum shape variations are computed as $\frac{\sqrt{2}}{r_0}\sum_{j=1}^{d/2} |b_{2j}|$.}\label{tab:maxshapevar}
\end{table}

Before looking at the results for the two model problems, we describe the neural network parameters used in our experiments. We have implemented our neural network in PyTorch \cite{pytorch}. The network architecture is fixed for all experiments, apart from the dimension of the input and output layers which depend on the case considered. It has depth $L=10$, the input layer has width $d$, the dimension of the parameter $\By\in\mathcal{P}_d$, the output layer has width the number of evaluation points $N_p$ and all hidden layers have width equal to $10$. As activation function we use the $0.2$-leaky ReLU. For optimizing the weights we use the Adam optimizer \cite{KL} with the default pyTorch parameters, learning rate $2\cdot 10^{-4}$ and weights initialized according to a uniform distribution $\sim\mathcal{U}(-a,a)$, $a=\frac{1}{\sqrt{10}}$ for all weights but those associated to the last layer, for which $a=\frac{1}{\sqrt{N_p}}$. We stopped the training after $50000$ epochs. Since we will not use a very large number of training samples, we run the optimization in full batch mode. In order to avoid the occurrence of local minima in the optimization procedure, for each case we have run the optimization algorithm $20$ times and from those we have selected the run delivering the lowest test error (see below for details on its computation). The loss function is given by the square of the $L^2$-relative error over the training samples, \eqref{eq:loss}-\eqref{eq:l2relative}. As error measure we consider therefore the square root of the loss. We measure the error of the network by computing the square root of the loss function using a number of test samples.  
We use $8192$ training samples and $2048$ test samples, for all cases. The number of training samples has been chosen such that difference between the error (square root of the loss) computed with the training and test samples was a fraction of the test error, so that our results are not affected (too much) by the sampling error. For the elliptic case and one point evaluation, where the sampling errors were the largest, using $8192$ samples gave an average difference in the error computed using the training and test samples of $6.99\%$ of the test error, with a variance of $0.17\%$.

The solution of \eqref{eq:ep} and \eqref{eq:hp} for every parameter realization has been implemented in Python using FEniCS \cite{fenics}, version 2019.1.0 of the Dolfin library.

\subsection{Elliptic interface problem}

We first provide the computation details of the solver used to compute the training and test samples of the quantity of interest. The mesh for the domain $D=(-1,1)^2$ consists of 53250 vertices, with a mesh size of 0.002 around the nominal interface, in order to resolve the shape variations, and of 0.03-0.04 in the other regions of the domain. The nominal radius is $r_0=0.5$. For the mapping \eqref{eq:phicircle}, we choose $\tilde{R}=0.875$, and the mollifier is given by 
\begin{equation}\label{eq:mollifier}
\chi(\lVert\hat{\Bx}\rVert)=\begin{cases}
\frac{\lVert\hat{\Bx}\rVert-\frac{r_0}{4}}{\frac{3r_0}{4}} & \text{if }\frac{r_0}{4}\leq\lVert\hat{\Bx}\rVert<r_0,\\
\frac{\tilde{R}-\lVert\hat{\Bx}\rVert}{\tilde{R}-r_0} & \text{if }r_0\leq\lVert\hat{\Bx}\rVert<\tilde{R},\\
0 & \text{otherwise}.
\end{cases}
\end{equation}
This mollifier is not smooth across the circles of radius $\frac{r_0}{4}$ and of radius $\tilde{R}$. However, the theory of the prevous sections still applies if we resolve these interfaces (as well as the nominal interface with radius $r_0$) with the mesh, as it is done in our computations. As right hand side in \eqref{eq:ep} we have taken $f(\Bx)=20+10\sin(x_1)-5\exp(x_1x_2)$, $\Bx\in D$. The variational problem \eqref{eq:ep} has been solved using first order finite elements, and the resulting algebraic system of equations has been solved with the conjugate gradient method using an algebraic multigrid preconditioner.

Tables \ref{tab:elliptic_onept} and \ref{tab:elliptic_64pt} show the square root of the loss \eqref{eq:loss}-\eqref{eq:l2relative} computed on the test samples, for different values of the contrast $\alpha_i$, different dimensions $d$ of the parameter space and different decays $p$ of the coefficients in \eqref{eq:radiusy}, when considering one and $64$ point evaluations, respectively. For all cases, we can see that the neural network achieves to provide a good surrogate for the point evaluation, with a test error in the square root of the loss ranging from $0.12$\textperthousand$\,$ to $0.90 $\textperthousand$\,$ for one point evaluation and from $0.40 \%$ to  $0.95 \%$ for $64$ points evaluations. 

\begin{table}
\begin{center}
 \begin{tabular}{||c|| c |c| c| c||} 
 \multicolumn{5}{c}{$\alpha_i=10$}\\
 \hline
\backslashbox{$\boldsymbol{p}$}{$\boldsymbol{d}$}  &\textbf{8} & \textbf{16} & \textbf{32} & \textbf{64} \\ [0.5ex] 
 \hline\hline
$\mathbf{1}$ & $0.48 $\textperthousand & $0.48 $\textperthousand & $0.62 $\textperthousand & $0.65 $\textperthousand\\
 \hline
$\mathbf{2}$& $0.18 $\textperthousand & $0.20 $\textperthousand & $0.21 $\textperthousand &  $0.27 $\textperthousand\\
  \hline
$\mathbf{3}$ & $0.13 $\textperthousand & $0.12 $\textperthousand & $0.22$\textperthousand & $0.20 $\textperthousand\\ 
  \hline
 \end{tabular}
  \begin{tabular}{||c|| c |c| c| c||} 
 \multicolumn{5}{c}{$\alpha_i=100$}\\
 \hline
\backslashbox{$\boldsymbol{p}$}{$\boldsymbol{d}$}  &\textbf{8} & \textbf{16} & \textbf{32} & \textbf{64} \\ [0.5ex] 
 \hline\hline
$\mathbf{1}$ & $0.59 $\textperthousand & $0.67 $\textperthousand & $0.75 $\textperthousand &  $0.90 $\textperthousand\\
 \hline
$\mathbf{2}$& $0.25 $\textperthousand & $0.22 $\textperthousand & $0.23 $\textperthousand & $0.24 $\textperthousand\\
  \hline
$\mathbf{3}$ & $0.15 $\textperthousand & $0.15 $\textperthousand &  $0.15 $\textperthousand & $0.23 $\textperthousand\\ 
  \hline
 \end{tabular}
 
 \vspace{0.3cm}
   \begin{tabular}{||c|| c |c| c| c||} 
 \multicolumn{5}{c}{$\alpha_i=1000$}\\
 \hline
\backslashbox{$\boldsymbol{p}$}{$\boldsymbol{d}$}  &\textbf{8} & \textbf{16} & \textbf{32} & \textbf{64} \\ [0.5ex] 
 \hline\hline
$\mathbf{1}$ & $0.63$\textperthousand &  $0.71$\textperthousand & $0.77$\textperthousand & $0.82$ \textperthousand\\
 \hline
$\mathbf{2}$& $0.22$\textperthousand &  $0.23$\textperthousand &  $0.28$\textperthousand & $0.27$ \textperthousand\\
  \hline
$\mathbf{3}$ & $0.13$\textperthousand &  $0.14$\textperthousand &  $0.16$\textperthousand & $0.17$\textperthousand\\ 
  \hline
 \end{tabular}

 \end{center}\caption{Elliptic interface problem: square root of the loss \eqref{eq:loss}-\eqref{eq:l2relative} for the point evaluation at $\Bx_0=(r_0,0)$, for different values of the contrast $\alpha_i$, different decays $p$ of the coefficients in \eqref{eq:radiusy} and different dimensions $d$ of the parameter space.}\label{tab:elliptic_onept}
\end{table}
\begin{table}
\begin{center}
\begin{tabular}{||c|| c |c| c| c||} 
  \multicolumn{5}{c}{$\alpha_i=10$}\\
 \hline
\backslashbox{$\boldsymbol{p}$}{$\boldsymbol{d}$}  &\textbf{8} & \textbf{16} & \textbf{32} & \textbf{64} \\ [0.5ex] 
 \hline\hline
$\mathbf{1}$ & $0.83 \%$ & $0.79 \%$ & $0.88 \%$ & $0.79 \%$ \\
 \hline
$\mathbf{2}$& $0.47 \%$ & $0.46 \%$  & $0.48 \%$ & $0.47 \%$ \\
  \hline
$\mathbf{3}$ & $0.42 \%$ & $0.51 \%$ & $0.50 \%$ & $0.41 \%$\\
  \hline
 \end{tabular}\hfill
  \begin{tabular}{||c|| c |c| c| c||} 
   \multicolumn{5}{c}{$\alpha_i=100$}\\
 \hline
\backslashbox{$\boldsymbol{p}$}{$\boldsymbol{d}$}  &\textbf{8} & \textbf{16} & \textbf{32} & \textbf{64} \\ [0.5ex] 
 \hline\hline
$\mathbf{1}$ & $0.91 \%$ & $0.93 \%$ & $0.93 \%$ & $0.94 \%$\\
 \hline
$\mathbf{2}$& $0.62 \%$ & $0.54 \%$ & $0.53 \%$ & $0.58 \%$\\
  \hline
$\mathbf{3}$ & $0.46 \%$ & $0.47 \%$ & $0.54 \%$ & $0.45 \%$ \\
  \hline
 \end{tabular}
 
 \vspace{0.3cm}
   \begin{tabular}{||c|| c |c| c| c||} 
\multicolumn{5}{c}{$\alpha_i=1000$}\\
 \hline
\backslashbox{$\boldsymbol{p}$}{$\boldsymbol{d}$}  &\textbf{8} & \textbf{16} & \textbf{32} & \textbf{64} \\ [0.5ex] 
 \hline\hline
$\mathbf{1}$ & $0.88 \%$ & $0.89 \%$  & $0.95 \%$ & $0.95 \%$\\
 \hline
$\mathbf{2}$& $0.54 \%$ & $0.55 \%$ & $0.55 \%$ & $0.56 \%$\\
  \hline
$\mathbf{3}$ & $0.56 \%$ & $0.40 \%$ & $0.42 \%$ & $0.50 \%$ \\
  \hline
 \end{tabular}
 \end{center}\caption{Elliptic interface problem: square root of the loss \eqref{eq:loss}-\eqref{eq:l2relative} for $64$ point evaluations at equispaced points on the circumference of radius $r_0$, for different values of the contrast $\alpha_i$, different decays $p$ of the coefficients in \eqref{eq:radiusy} and different dimensions $d$ of the parameter space.}\label{tab:elliptic_64pt}
\end{table}

\begin{table}
\begin{center}
 \begin{tabular}{|c| c |c| c |c|} 
  \multicolumn{5}{c}{$p=1,\, d=32,\, \alpha_i=100$}\\
 \hline
 $x_1=0.40$ & $x_1=0.45$ &  $x_1=0.5$  & $x_1=0.55$ & $x_1=0.60$ \\ [0.5ex] 
 \hline\hline
$0.64 $\textperthousand &  $0.67 $\textperthousand & $0.75 $\textperthousand & $0.82 $\textperthousand  & $0.75 $\textperthousand\\
  \hline
 \end{tabular}\hfill
 \end{center}\caption{Elliptic interface problem: square root of the loss \eqref{eq:loss}-\eqref{eq:l2relative} for one point evaluation and different evaluation points $\Bx_0=(x_1,0)$. The decay $p$ of the coefficients has been set to $p=1$, the dimension $d$ of the parameter space to $d=32$ and the contrast to $\alpha_i=100$.}\label{tab:elliptic_disclocation}
\end{table}

\begin{figure}
\centering
\includegraphics[scale=0.35]{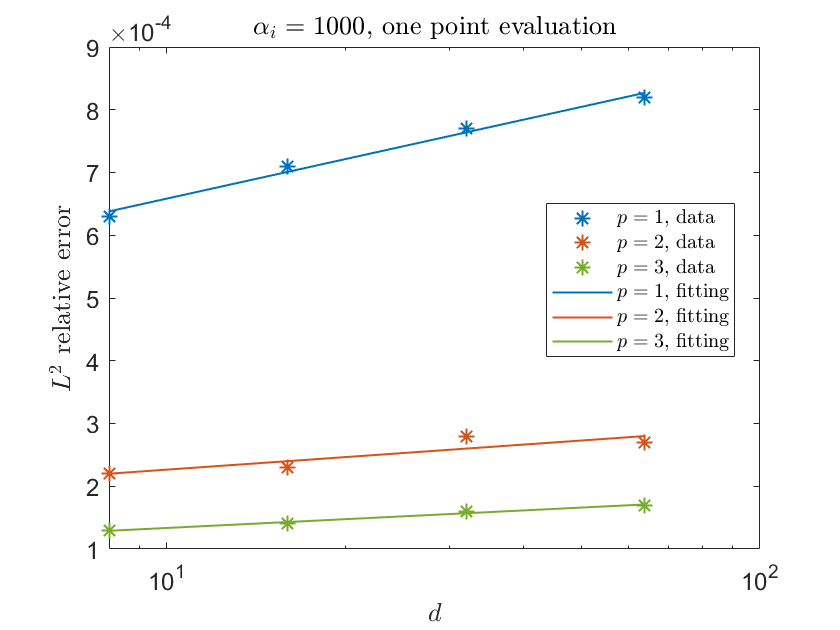}
\includegraphics[scale=0.35]{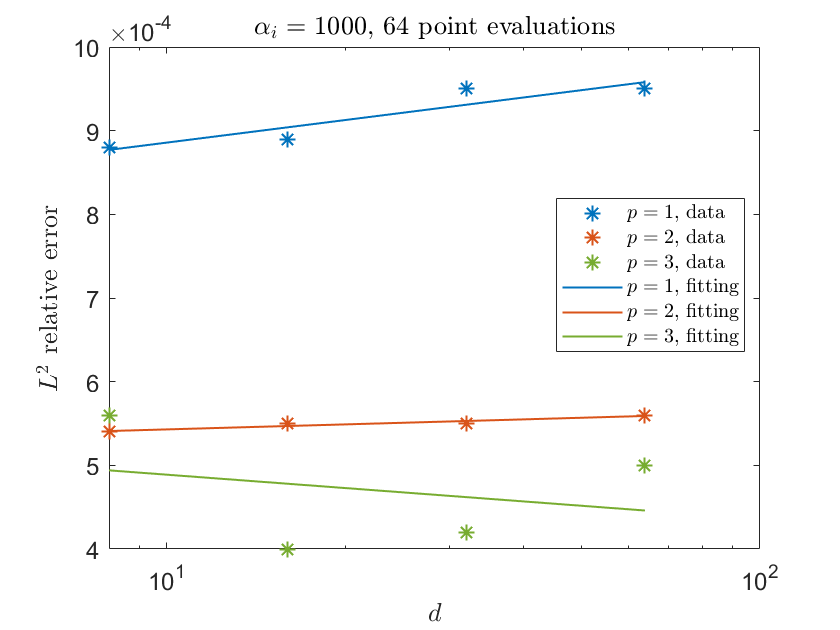}
\caption{Elliptic interface problem: square root of the loss \eqref{eq:loss}-\eqref{eq:l2relative} in dependence of the dimension $d$ of the parameter space, for different decays $p$ of the coefficients in \eqref{eq:radiusy}, one (left) and $64$(right) point evaluations. The contrast has been fixed to $\alpha_i=1000$. The markers denote the data points, while the lines are the least square fitting.}\label{fig:elliptic_d}
\end{figure}

\begin{figure}
\centering
\includegraphics[scale=0.35]{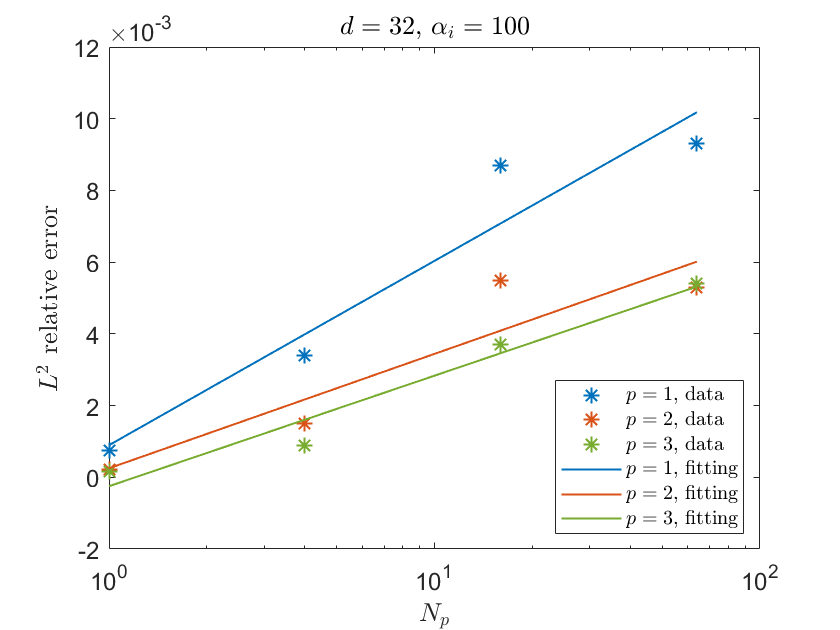}
\includegraphics[scale=0.35]{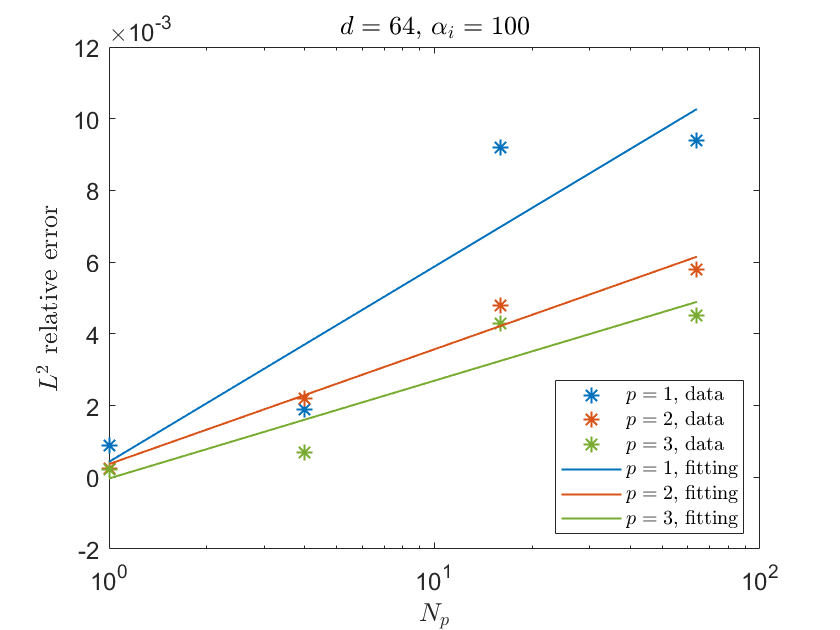}\caption{Elliptic interface problem: square root of the loss \eqref{eq:loss}-\eqref{eq:l2relative} in dependence of the number of point evaluations at equispaced points on the circumference of radius $r_0$, for different decays $p$ of the coefficients in \eqref{eq:radiusy}, dimensions $d=32$ (left) and $d=64$ (right) of the parameter space. The contrast has been fixed to $\alpha_i=100$. The markers denote the data points, while the lines are the least square fitting.}\label{fig:elliptic_Np}
\end{figure}

To better observe the increase of the error with the dimension of the parameter space, in Figure \ref{fig:elliptic_d} we plot the error versus the dimension $d$ of the parameter space, for different values of $p$, and for $1$ and $64$ point evaluations. The contrast there has been fixed to $\alpha_i=1000$. We remind that the value of $p$ is connected with the intrinsic dimensionality of the problem, in the sense that larger values of $p$ make the higher dimensions of the parameter space less important compared to the first dimensions. Indeed, from Figure \ref{fig:elliptic_d} we see that the dependence of the error with respect to the dimension is very mild and indeed a significative dependence of the type $O(\log d)$ can be observed for $p=1$ only. In particular, for $p=2,3$ the error is $O(1)$ with respect to the dimension, and for $64$ points evaluations we can see variations in the error which are probably due mainly to the optimization error and which are causing a fitting with a line of slightly negative slope in the left plot of Figure \ref{fig:elliptic_d}.

Tables \ref{tab:elliptic_onept} and \ref{tab:elliptic_64pt} also tell us how the test error depends on the contrast. In particular, we see that the error depends very mildly on the contrast and a clear dependence can be observed essentially only when comparing the results for $\alpha_i=10$ with those for $\alpha_i=100,1000$ for $p=1$, for both one and $64$ point evaluations. 

We now look at the dependence of the error on the surfaces of non-smoothness. In Table \ref{tab:elliptic_disclocation} we have considered one point evaluation, decay $p=1$, dimension $d=32$ of the parameter space and contrast $\alpha_i=100$, and we have computed the error for different locations of the evaluation point $\Bx_0=(x_1,0)$ on the horizontal line through the origin, that is varying $x_1$. When $x_1<r_0=0.5$, the point will be inside the scatterer for most realizations. When $x_1>r_0$, it will be outside the scatterer for the majority of the interface realizations. Table \ref{tab:elliptic_disclocation} shows that the test error tends to be larger for the case $x_1\geq r_0$ than when $x_1<r_0$. We notice that, since $\alpha_i>1$ and the diffusion coefficient is $1$ outside the scatterer, the solution in the physical space has larger gradients outside the scatterer than when it is inside. This behavior in the physical space affects the dependence of the point evaluation on the parameter, which will have therefore higher gradients when the point is outside the scatterer for most realizations. These higher gradients could be the reason why we observe larger errors for $x_1> r_0$. To observe the dependence of the error on the number of surfaces of non-smoothness as in \eqref{eq:hyperplane}, in Figure \ref{fig:elliptic_Np}, for different values of $p$, $\alpha_i=100$, $d=32$ (left plot) and $d=64$ (right plot), we show the test error in dependence of the number of point evaluations, therefore in dependence on the number of hyperplanes of discontinuity. We note that, since the dimension of the output depends on the number of discontinuities, the neural networks for different number of point evaluations have a slightly different number of nodes and weights. However, this difference is negligible compared to the size of the networks, so that the results for different point evaluations can be still considered to be comparable to each other. From Figure \ref{fig:elliptic_Np} we see that, as expected, the error increases with an increasing number of discontinuities. The dependence of the error seems to be $O(\log N_p)$, with a constant depending on $p$. Such increase of the error is milder than the linear dependence predicted by the theory, see Corollary \ref{cor:morepts}.

\subsection{Helmholtz transmission problem}

The computation details of the solver used to compute the training and test samples of the quantity of interest are the following. The domain $D$ is a circle of radius $R=0.055$ and the nominal radius is $r_0=0.01$. With the exception of the experiments where we vary the frequency, we have always used $\kappa_o=\frac{2 00\pi}{3}$ as wavenumber for the incoming wave. In order to approximate the Sommerfeld radiation condition \eqref{eq:somm}, we have used a circular Perfectly Matched Layer (PML, see for instance \cite{CM}) of thickness $0.02$ and damping coefficient $0.5$. Apart from the experiments where we vary the frequency, we have always used a mesh with 124014 elements, with a mesh size of about 0.00016 around the nominal interface with radius $r_0$ and between 0.0006 and 0.00085 in the PML layer. In the domain mapping, we have used the same mollifier as in \eqref{eq:mollifier}, with $\tilde{R}=0.055$ and by resolving the interface with radius $\frac{r_0}{4}$ with the mesh. The variational problem \eqref{eq:varformhp} has been solved using first order finite elements for the real and imaginary part of the solution, and the resulting algebraic system of equations has been solved using a direct solver. 

Tables \ref{tab:helmholtz_onept} shows the test error for one point evaluation at $\Bx_0=(r_0,0)$ and different values of the contrast $\alpha_i$, different dimensions $d$ of the parameter space and different decays $p$ of the coefficients in \eqref{eq:radiusy}. There we have fixed $\kappa_i/\kappa_0=0.8$. Table \ref{tab:helmholtz_64pt} shows the same results for the simulaneous evaluation of the solution in $64$ equispaced points on the circumference with radius $r_0$. From the two tables we can see that in all test cases the neural network provides a good surrogate for the point evaluation. Comparing with Tables \ref{tab:elliptic_onept} and \ref{tab:elliptic_64pt}, we also see that the errors are larger for the Helmholtz transmission problem. Figure \ref{fig:helmholtz_d} shows the error in dependence of the dimension $d$ for the case $\alpha_i=1000$ in Table \ref{tab:helmholtz_onept} (left plot) and in Table \ref{tab:helmholtz_64pt} (right plot). From Tables  \ref{tab:helmholtz_onept} and \ref{tab:helmholtz_64pt} and from Figure \ref{fig:helmholtz_d} we can see that the error increases slighlty as the dimension $d$ of the parameter space increases. Figure \ref{fig:helmholtz_d} shows that the dependence of the error on the dimension is $O(\log d)$, with a slope depending on $p$. The slopes for $p=2$ and $p=3$ are so small that the dependence is almost $O(1)$. Tables  \ref{tab:helmholtz_onept} and \ref{tab:helmholtz_64pt} also show a mild increase of the error as the contrast $\alpha_i$ increases. We remind that $\alpha_i$ determines the jump of the first derivative of the solution in the physical space across the nominal interface.

\begin{table}
\begin{center}
 \begin{tabular}{||c|| c |c| c| c||} 
  \multicolumn{5}{c}{$\alpha_i=10, \kappa_i/\kappa_o=0.8$}\\
 \hline
\backslashbox{$\boldsymbol{p}$}{$\boldsymbol{d}$}  &\textbf{8} & \textbf{16} & \textbf{32} & \textbf{64} \\ [0.5ex] 
 \hline\hline
$\mathbf{1}$ & $0.45$\textperthousand & $0.92$\textperthousand & $1.18$\textperthousand & $1.36$\textperthousand\\
 \hline
$\mathbf{2}$& $0.15$\textperthousand & $0.18$\textperthousand & $0.20$\textperthousand & $0.20$\textperthousand \\
  \hline
$\mathbf{3}$ & $0.09$\textperthousand  & $0.08$\textperthousand & $0.10$\textperthousand & $0.10$\textperthousand \\
  \hline
 \end{tabular}\hfill
  \begin{tabular}{||c|| c |c| c| c||} 
  \multicolumn{5}{c}{$\alpha_i=100, \kappa_i/\kappa_o=0.8$}\\
 \hline
\backslashbox{$\boldsymbol{p}$}{$\boldsymbol{d}$}  &\textbf{8} & \textbf{16} & \textbf{32} & \textbf{64} \\ [0.5ex] 
 \hline\hline
$\mathbf{1}$ & $0.63$\textperthousand & $1.38$\textperthousand & $1.71$\textperthousand & $2.01$\textperthousand\\
 \hline
$\mathbf{2}$& $0.22$\textperthousand & $0.28$\textperthousand & $0.25$\textperthousand & $0.30$\textperthousand \\
  \hline
$\mathbf{3}$ & $0.12$\textperthousand  & $0.13$\textperthousand & $0.15$\textperthousand & $0.13$\textperthousand \\
  \hline
 \end{tabular}
 
 \vspace{0.3cm}
   \begin{tabular}{||c|| c |c| c| c||} 
  \multicolumn{5}{c}{$\alpha_i=1000, \kappa_i/\kappa_o=0.8$}\\
 \hline
\backslashbox{$\boldsymbol{p}$}{$\boldsymbol{d}$}  &\textbf{8} & \textbf{16} & \textbf{32} & \textbf{64} \\ [0.5ex] 
 \hline\hline
$\mathbf{1}$ & $0.64$\textperthousand & $1.34$\textperthousand & $1.82$\textperthousand & $2.15$\textperthousand\\
 \hline
$\mathbf{2}$& $0.21$\textperthousand & $0.24$\textperthousand & $0.28$\textperthousand & $0.30$\textperthousand \\
  \hline
$\mathbf{3}$ & $0.10$\textperthousand  & $0.13$\textperthousand & $0.14$\textperthousand & $0.16$\textperthousand \\
  \hline
 \end{tabular}
 \end{center}\caption{Helmholtz transmission problem: square root of the loss \eqref{eq:loss}-\eqref{eq:l2relative} for the point evaluation at $\Bx_0=(r_0,0)$, for different values of the contrast $\alpha_i$, different decays $p$ of the coefficients in \eqref{eq:radiusy} and different dimensions $d$ of the parameter space. We have fixed $\frac{\kappa_i}{\kappa_o}=0.8.$}\label{tab:helmholtz_onept}
 \end{table}

\begin{table}
\begin{center}
 \begin{tabular}{||c|| c |c| c| c||} 
  \multicolumn{5}{c}{$\alpha_i=10, \kappa_i/\kappa_o=0.8$}\\
 \hline
\backslashbox{$\boldsymbol{p}$}{$\boldsymbol{d}$}  &\textbf{8} & \textbf{16} & \textbf{32} & \textbf{64} \\ [0.5ex] 
 \hline\hline
$\mathbf{1}$ &$1.79 \%$  & $2.80 \%$ & $2.57 \%$ & $2.59 \%$\\
 \hline
$\mathbf{2}$& $1.09 \%$ & $1.19 \%$ & $1.13 \%$ & $1.52 \%$ \\
  \hline
$\mathbf{3}$ &  $0.75 \%$ & $0.80 \%$ & $0.70 \%$ & $0.73 \%$ \\
  \hline
 \end{tabular}\hfill
  \begin{tabular}{||c|| c |c| c| c||} 
  \multicolumn{5}{c}{$\alpha_i=100, \kappa_i/\kappa_o=0.8$}\\
 \hline
\backslashbox{$\boldsymbol{p}$}{$\boldsymbol{d}$}  &\textbf{8} & \textbf{16} & \textbf{32} & \textbf{64} \\ [0.5ex] 
 \hline\hline
$\mathbf{1}$ & $2.15 \%$ & $3.21 \%$ & $3.72 \%$ & $3.75 \%$\\
 \hline
$\mathbf{2}$& $1.23 \%$  & $1.31 \%$ & $1.33 \%$ & $1.34 \%$\\
  \hline
$\mathbf{3}$ & $0.86 \%$ & $0.83 \%$ & $0.89 \%$ &  $0.91 \%$\\
  \hline
 \end{tabular}
 
 \vspace{0.3cm}
   \begin{tabular}{||c|| c |c| c| c||} 
  \multicolumn{5}{c}{$\alpha_i=1000, \kappa_i/\kappa_o=0.8$}\\
 \hline
\backslashbox{$\boldsymbol{p}$}{$\boldsymbol{d}$}  &\textbf{8} & \textbf{16} & \textbf{32} & \textbf{64} \\ [0.5ex] 
 \hline\hline
$\mathbf{1}$ & $2.28 \%$ & $3.13 \%$ & $4.22 \%$ & $3.98 \%$\\
 \hline
$\mathbf{2}$& $1.33 \%$ & $1.35 \%$ & $1.27 \%$ & $1.34 \%$\\
  \hline
$\mathbf{3}$ & $0.81 \%$ & $0.90 \%$ & $0.95 \%$ & $0.89 \%$\\
  \hline
 \end{tabular}
 \end{center}\caption{Helmholtz transmission problem: square root of the loss \eqref{eq:loss}-\eqref{eq:l2relative} for $64$ point evaluations at equispaced points on the circumference of radius $r_0$, for different values of the contrast $\alpha_i$, different decays $p$ of the coefficients in \eqref{eq:radiusy} and different dimensions $d$ of the parameter space. We have fixed $\frac{\kappa_i}{\kappa_o}=0.8.$}\label{tab:helmholtz_64pt}
 \end{table}
 
\begin{figure}
\centering
\includegraphics[scale=0.35]{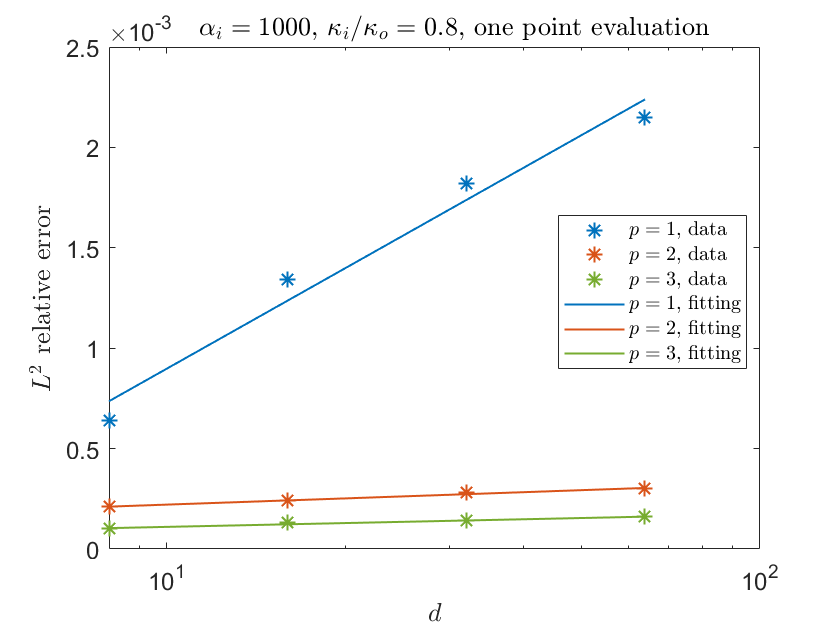}
\includegraphics[scale=0.35]{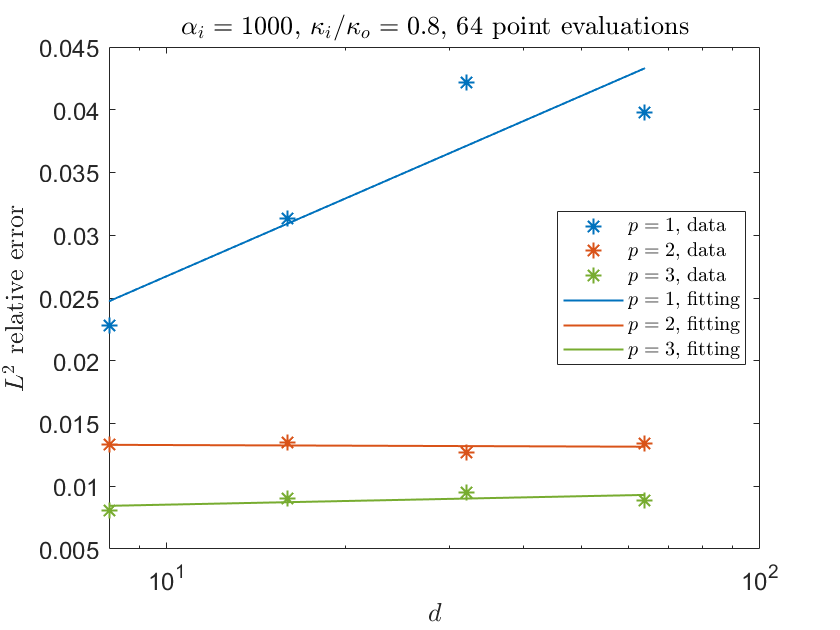}
\caption{Helmholtz interface problem: square root of the loss \eqref{eq:loss}-\eqref{eq:l2relative} in dependence of the dimension $d$ of the parameter space, for different decays $p$ of the coefficients in \eqref{eq:radiusy}, one (left) and $64$(right) point evaluations. The contrasts have been fixed to $\alpha_i=1000$ and $\frac{\kappa_i}{\kappa_0}=0.8$. The markers denote the data points, while the lines are the least square fitting.}\label{fig:helmholtz_d}
\end{figure}

In order to study the behavior of the error in dependence of the jumps of the derivatives of the solution in physical space across the nominal interface, we consider now $\alpha_i=1$ and we vary the contrast $\frac{\kappa_i}{\kappa_o}$. The results are shown in Table \ref{tab:helmholtz_varyn}, for one (left plot) and $64$ (right plot) point evaluations, where we have fixed the dimension of the parameter space to $d=16$. Differently from the previous cases, since now $\alpha_i=1$, the solution in the physical space is not only $C^0$ but $C^1$ across the nominal interface, and therefore the dependence on the stochastic parameter is also of class $C^1$, see Proposition \ref{prop:contyhelm}. Comparing the numbers in Table \ref{tab:helmholtz_varyn} for $\frac{\kappa_i}{\kappa_o}=0.8$ with those in Tables \ref{tab:helmholtz_onept} and \ref{tab:helmholtz_64pt}, we see that the errors are lower in the case of $C^1$ dependence of the quantity of interest with respect to the parameter. Table \ref{tab:helmholtz_varyn} shows also that the errors increase as we increase  the contrast $\frac{\kappa_i}{\kappa_o}$.
 
 \begin{table}
\begin{center}
 \begin{tabular}{||c|| c |c| c||} 
  \multicolumn{4}{c}{$d=16, \alpha_i=1$}\\ 
 \hline
\backslashbox{$\boldsymbol{p}$}{$\frac{\boldsymbol{\kappa_i}}{\boldsymbol{\kappa_o}}$}  & $0.8$ & $0.08$ & $0.008$ \\ [0.5ex] 
 \hline\hline
$\mathbf{1}$ &$0.21$\textperthousand  & $1.62$\textperthousand & $1.86$\textperthousand\\
 \hline
$\mathbf{2}$& $0.10$\textperthousand  &  $0.51$\textperthousand & $0.53$\textperthousand \\
  \hline
$\mathbf{3}$ & $0.08$\textperthousand  & $0.26$\textperthousand & $0.34$\textperthousand \\
  \hline
 \end{tabular}\hspace{1cm}
  \begin{tabular}{||c|| c |c| c||} 
  \multicolumn{4}{c}{$d=16, \alpha_i=1$}\\ 
 \hline
\backslashbox{$\boldsymbol{p}$}{$\frac{\boldsymbol{\kappa_i}}{\boldsymbol{\kappa_o}}$}  & $0.8$ & $0.08$ & $0.008$ \\ [0.5ex] 
 \hline\hline
$\mathbf{1}$ & $0.55 \%$  &  $1.92 \%$ &  $2.10 \%$\\
 \hline
$\mathbf{2}$&  $0.26 \%$  &  $0.86 \%$ &  $0.98 \%$ \\
  \hline
$\mathbf{3}$ & $0.28 \%$  &  $0.60 \%$ &  $0.57 \%$ \\
  \hline
 \end{tabular}
 \end{center}\caption{Helmholtz transmission problem: square root of the loss \eqref{eq:loss}-\eqref{eq:l2relative} for one point evaluation at $\Bx_0=(r_0,0)$ (left) and $64$ point evaluations (right) when $\alpha_i=1$, for different values of the contrast $\frac{\kappa_i}{\kappa_o}$ and different decays $p$ of the coefficients in \eqref{eq:radiusy}. The dimension $d$ of the parameter space has been fixed to $d=16$.}\label{tab:helmholtz_varyn}
 \end{table}

Fixing now $\alpha_i=100$ and $\frac{\kappa_i}{\kappa_o}=0.8$, we examine now the dependence of the error on the number $N_p$ of point evaluations, that is on the number of surfaces of non-smoothness. This is shown in Figure \ref{fig:Helmholtz_Np}, for $d=32$ (left plot) and $d=64$ (right plot). As in the elliptic case, we observe a logarithmic increase of the error with the number of point evaluations, with a slope depending on the decay $p$ of the coefficients in the radius expansion. In particular, also for the Helmholtz transmission problem the increase of the error is milder than predicted by the theory, see Corollary \ref{cor:morepts}. 
 

\begin{figure}
\centering
\includegraphics[scale=0.35]{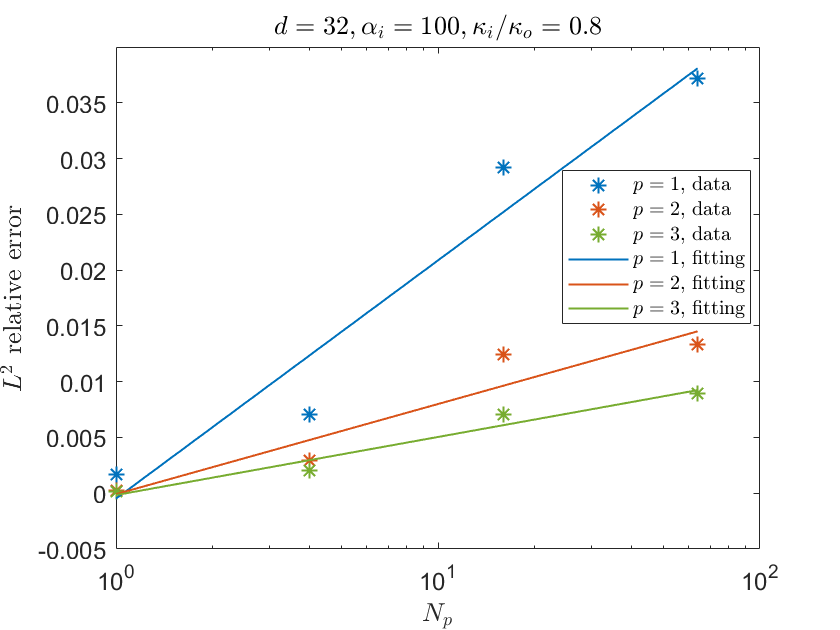}
\includegraphics[scale=0.35]{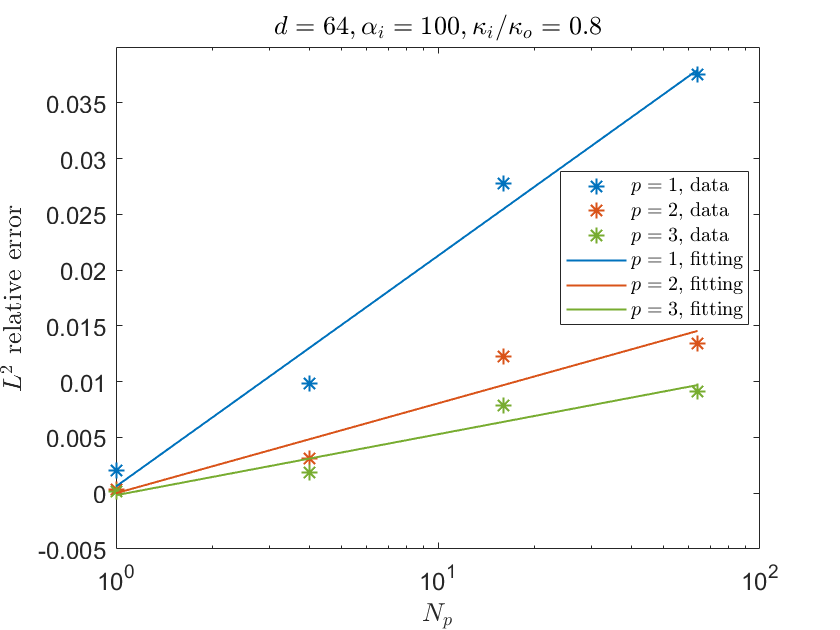}\caption{Helmholtz transmission problem: square root of the loss \eqref{eq:loss}-\eqref{eq:l2relative} in dependence of the number of point evaluations at equispaced points on the circumference of radius $r_0$, for different decays $p$ of the coefficients in \eqref{eq:radiusy}, dimensions $d=32$ (left) and $d=64$ (right) of the parameter space. The contrasts have been fixed to $\alpha_i=100$ and $\frac{\kappa_i}{\kappa_o}=0.8$. The markers denote the data points, while the lines are the least square fitting.}\label{fig:Helmholtz_Np}
\end{figure}

Finally, we study the dependence of the error on the frequency of the incoming wave, that is on the wavenumber $\kappa_o$ in free space. So far we have considered $\kappa_o=k_o:=\frac{200 \pi}{3}$. We now compare this case with the cases $\kappa_o=2k_0$ and $\kappa_o=3k_0$. For the case $\kappa_o=2k_0$ we have considered a finer mesh with $350582$ elements, and for $\kappa_0=3k_0$ we have taken an even finer mesh with $468682$ elements, so that the number of finite elements per wavelength is not too different for the different test cases. We fix the other parameters to $\alpha_i=100$, $\frac{\kappa_i}{\kappa_o}=0.8$ and $d=16$. The results are shown in Table \ref{tab:varyfreq}, for one (left plot) and $64$ (right plot) point evaluations. There we clearly see that the error is very sensitive to the wavenumber, with an increase which seems to be polynomial in the wavenumber. A possible explanation for such behavior is the following. The larger the outer wavenumber, the more oscillatory in the solution in the physical space; such space dependence intertwines with the parameter dependence, resulting in a more oscillatory behavior of the quantity of interest in the parameter space; such oscillatory behavior is harder to approximate by the neural network, leading to larger errors.

 \begin{table}
\begin{center}
 \begin{tabular}{||c|| c |c| c||} 
  \multicolumn{4}{c}{$\alpha_i=100, \kappa_i/\kappa_o = 0.8$}\\ 
 \hline
\backslashbox{$\boldsymbol{p}$}{$\boldsymbol{\kappa}$}  & $k_0$ & $2k_0$ & $3k_0$ \\ [0.5ex] 
 \hline\hline
$\mathbf{1}$ & $1.38$\textperthousand  & $2.04$\textperthousand & $4.31$\textperthousand\\
 \hline
$\mathbf{2}$& $0.28$\textperthousand  & $0.45$\textperthousand & $0.73$\textperthousand \\
  \hline
$\mathbf{3}$ & $0.13$\textperthousand  & $0.20$\textperthousand & $0.33$\textperthousand \\
  \hline
 \end{tabular}\hspace{1cm}
  \begin{tabular}{||c|| c |c| c||} 
  \multicolumn{4}{c}{$\alpha_i=100, \kappa_i/\kappa_o = 0.8$}\\ 
 \hline
\backslashbox{$\boldsymbol{p}$}{$\boldsymbol{\kappa}$}  & $k_0$ & $2k_0$ & $3k_0$ \\ [0.5ex] 
 \hline\hline
$\mathbf{1}$ & $3.21 \%$  & $5.57 \%$ & $9.20 \%$\\
 \hline
$\mathbf{2}$& $1.31 \%$  & $2.18 \%$ & $3.62 \%$ \\
  \hline
$\mathbf{3}$ & $0.83 \%$  & $1.42 \%$ & $2.09 \%$ \\
  \hline
 \end{tabular}
 \end{center}\caption{Helmholtz transmission problem: square root of the loss \eqref{eq:loss}-\eqref{eq:l2relative} in dependence of the outer wavenumber, for different decays $p$ of the coefficients in \eqref{eq:radiusy}, one (left) and $64$ (right) point evaluations on the circumference of radius $r_0$. Here $k_0=\frac{2 00\pi }{3}$. The contrasts have been fixed to $\alpha_i=100$ and $\frac{\kappa_i}{\kappa_o}=0.8$, and the dimension of the parameter space is $d=16$.}\label{tab:varyfreq}
 \end{table}
 
\subsection{Comments on the numerical results}

The observations deduced from the numerical experiments can be summarized as follows. 

\begin{itemize}
\item \textsl{Dimensionality}: the behavior of the error as the dimension $d$ of the parameter space increases depends on the decay $p$ of the coefficients in \eqref{eq:radiusy}. In particular, in both model problems we have observed a behavior $O(\log d)$ for $p=1$ and essentially $O(1)$ for $p=2,3$. The latter is not surprising, especially for $p=3$, when comparing the very little contribution of the higher dimensions to the maximum shape variations in Table \ref{tab:maxshapevar}.
\item \textsl{Model problem and modeling parameters}: the error of the deep learning surrogate depends on:
\begin{itemize}
\item \textsl{model problem}: in general we have seen that the error is larger for the Helmholtz case than in the elliptic case; furthermore, for the Helmholtz transmission problem we have seen that the error is very sensitive to the frequency and increases as the frequency increases;
\item \textsl{contrast}: the dependence of the error on the jump in the first derivatives, that is on the contrast $\alpha_i$, is in general very mild and it depends on $p$; in particular, we have observed larger errors for $p=1$ and $\alpha_i=100,1000$ compared to $p=1$ and $\alpha_i=10$, especially in the Helmholtz problem, while for $p=2,3$ we could not see a clear dependence of the error on the contrast;
\item \textsl{type of discontinuity}: for the Helmholtz transmission problem, we have seen that the error is smaller when $\alpha_i=1$ compared to when $\alpha_i\neq 1$, while keeping all other parameters fixed, that is, the error is smaller in case of $C^1$-dependence of the quantity of interest with respect to the stochastic parameter; for $\alpha_i=1$, when we then increase the contrast $\frac{\kappa_i}{\kappa_o}$, the error increases. 
\end{itemize}
\item \textsl{Number and location of surfaces of non-smoothness}: in both model problems, we have seen a behavior of the error of the kind $O(\log N_p)$, when $N_p$ is the number of evaluation points, that is the number of surfaces of non-smoothness. Such increase of the error is milder than what predicted by the theory in Corollary \ref{cor:morepts}. Moreover, for the elliptic case and one point evaluation, we have seen that the errors tends to be larger when the point is outside of the scatterer for most of the realizations than when it is inside for most realizations, and this could be due to the fact that the gradient with repect to the parameter is larger on the side of the parameter space for which the point is outside the scatterer.
\end{itemize}

\section{Conclusions}

The point evaluation of the solution to an interface problem with stochastic interface is known to depend non-smoothly on the parameter modeling the shape variations, posing a challenge to the construction of a surrogate for the parameter-to-quantity-of-interest map, especially in high dimensions. In this work, we have proposed to use deep neural networks to build such a surrogate. Based on previous results on neural network approximation, we have provided a theoretical justification for why we expect neural networks to be good surrogates for elliptic and Helmholtz interface problems. Numerical results confirm the good performance of some deep networks, with an error which increases logarithmically with the number of point evaluations, that is with the number of surfaces of non-smoothness. For our model problems the neural networks show also very clearly not to suffer from the curse of dimensionality, with a dependence of the error on the dimension of the parameter space ranging from $O(\log d)$ to $O(1)$. Furthermore, we have investigated numerically the dependence of the performance on various modeling parameters, such as the contrast and, for the Helmholtz problem, the frequency.

\section*{Acknowledgements}

The author would like to thank Ralf Hiptmair for having encouraged her in writing this article and for comments on a draft of this paper.

\bibliographystyle{siam}
\bibliography{references}
\end{document}